\documentclass[a4paper,11pt,notitlepage]{article}
 \usepackage[english]{babel}
  \usepackage{amsmath,amsfonts,amssymb,amsthm,relsize,setspace,nicefrac, yhmath,amscd}
  \usepackage{amsrefs, mathrsfs, enumerate} 
  \usepackage{authblk}
 \usepackage[colorlinks, linkcolor=red, citecolor=blue, urlcolor=blue, hypertexnames=true]{hyperref}
  \usepackage{graphicx} \usepackage{subcaption}
  \usepackage{hyperref} %%Warning: when you first run your tex
\usepackage[active]{srcltx} %%allows you to jump from your editor
 \usepackage{color,soul} %%allows you to use the \st{...} command, which
\newtheorem{theorem}{Theorem}[section]

\newtheorem{lemma}[theorem]{Lemma}
\newtheorem{proposition}[theorem]{Proposition}

\theoremstyle{definition}

\newtheorem{remark}{Remark}

\newcommand{\be}{\begin{equation}}
\newcommand{\bel}[1]{\begin{equation}\label{#1}}
\newcommand{\ee}{\end{equation}}
\newcommand{\barr}{\begin{eqnarray}}
\newcommand{\earr}{\end{eqnarray}}
\newcommand{\bars}{\begin{eqnarray*}}
\newcommand{\ears}{\end{eqnarray*}}

\newtheorem{subn}{\name}

\newcommand{\bsn}[1]{\def\name{#1}\begin{subn}}
\newcommand{\esn}{\end{subn}}
\newtheorem{sub}{\name}[section]

\newcommand{\bs}{\begin{sub}}
\newcommand{\es}{\end{sub}}
\newcommand{\bth}[1]{\def\name{Theorem}
\begin{sub}\label{t:#1}}
\newcommand{\blemma}[1]{\def\name{Lemma}
\begin{sub}\label{l:#1}}
\newcommand{\bcor}[1]{\def\name{Corollary}
\begin{sub}\label{c:#1}}
\newcommand{\bdef}[1]{\def\name{Definition}
\begin{sub}\label{d:#1}}
\newcommand{\bprop}[1]{\def\name{Proposition}
\begin{sub}\label{p:#1}}

\newcommand{\BA}{\begin{array}}
\newcommand{\EA}{\end{array}}
\newcommand{\BAN}{\renewcommand{\arraystretch}{1.2}
\setlength{\arraycolsep}{2pt}\begin{array}}
\newcommand{\BAV}[2]{\renewcommand{\arraystretch}{#1}
\setlength{\arraycolsep}{#2}\begin{array}}
\newcommand{\BSA}{\begin{subarray}}
\newcommand{\ESA}{\end{subarray}}

\newcommand{\BAL}{\begin{aligned}}
\newcommand{\EAL}{\end{aligned}}
\newcommand{\BALG}{\begin{alignat}}
\newcommand{\EALG}{\end{alignat}}
\newcommand{\BALGN}{\begin{alignat*}}
\newcommand{\EALGN}{\end{alignat*}}
\def\angb<#1>{\langle #1 \rangle}%% angle bracket
\newcommand{\supp}{\opname{supp}}

\def\R{\mathbb{R}}

\let\e=\varepsilon

\let\.=\cdot
\let\0=\emptyset

\def\supp{\text{\rm supp}}

\numberwithin{equation}{section}

\theoremstyle{definition}

\let\e=\varepsilon

\let\.=\cdot
\let\0=\emptyset

\def\supp{\text{\rm supp}}

\newenvironment{formula}[1]{\begin{equation}\label{eq:#1}}
                       {\end{equation}\noindent}

\def\Fi#1{\begin{formula}{#1}}
\def\Ff{\end{formula}\noindent}

\setstcolor{red}

\title{A unified clasification of Liouville properties and nontrivial solution for fractional elliptic equations with general Hénon-type superquadratic and gradient growth}

\author[1]{Hoang-Hung Vo \thanks{Corresponding author, Email:  vhhung@sgu.edu.vn}}
\affil[1]{Faculty of Mathematics and Applications, Saigon University, 273 An Duong Vuong st., Ward Choquan, Ho Chi Minh City, Vietnam}
 
 \date{}

\begin{document}
\maketitle

%%%%%%%%%%%%%%
%%Setting up the TITLE and 

%\thanks{This research is funded by Vietnam National Foundation for Science and Technology Development (NAFOSTED) under grant number 101.02-2023.12}

\begin{abstract} 
We investigate Liouville-type results, existence, uniqueness and symmetry to the solution of nonlinear nonlocal elliptic equations of the form
\[
Lu = |x|^{\gamma}\,H(u)\,G(\nabla u), \qquad x\in\R^n,
\]
where $L$ is a symmetric, translation-invariant, uniformly elliptic integro--differential operator of order $2s\in(0,2)$, and $H,G$ satisfy general structural and growth conditions. A unified analytical framework is developed to identify the precise critical balance $\gamma+p=2s$, which separates the supercritical, critical, and subcritical situations. In the supercritical case $\gamma+p>2s$, the diffusion dominates the nonlinear term and every globally defined solution with subcritical growth must be constant; in the critical case $\gamma+p=2s$, all bounded positive solutions are constant, showing that the nonlocal diffusion prevents the formation of nontrivial equilibria; in the subcritical case $\gamma+p<2s$,  we are able construct a unique, positive, radially symmetric, and monotone entire solution with explicit algebraic decay
\[
u(x)\sim (1+|x|^2)^{-\beta}, \qquad \beta=\frac{2s+\gamma-p}{1-p}>0.
\]
The proofs rely on new nonlocal analytical techniques, including quantitative cutoff estimates for general integro--differential kernels, a fractional Bernstein-type transform providing pointwise gradient control, and moving plane and sliding methods formulated in integral form to establish symmetry and uniqueness. The current investigation provides an equivalent and unifying contribution to Liouville properties 
and related existence results, comparable to the recent deep studies of Chen--Dai--Qin~\cite{Chen2023} and Biswas--Quaas--Topp~\cite{Biswas2025} on this direction.
\end{abstract}

 
%\subjclass[2010]{Primary 35B50, 47G20; secondary 35J60}

%\date{January 1, 2001 and, in revised form, June 22, 2001.}


\tableofcontents

\section{Introduction and Main results}

Liouville-type theorems play a central role in the qualitative theory of elliptic
and parabolic partial differential equations, providing fundamental rigidity
properties for global solutions of nonlinear problems.  
They describe circumstances under which every entire (globally defined)
solution to a nonlinear equation must be constant, and hence connect
local analytic behavior with global geometric and scaling features of the equation.
Originally established for the classical Laplacian, such results now constitute
a cornerstone in modern PDE analysis, touching topics as diverse as
nonexistence, blow-up, symmetry, and regularity theory. In this paper, we aim to completely classify a \emph{unified Liouville theorem}
for a broad class of nonlinear nonlocal elliptic equations with general Hénon-type superquadratic and gradient growth read as follow :
\begin{equation}\label{eq:main}
 Lu(x)=|x|^{\gamma} H(u(x)) G(\nabla u(x)), \qquad x\in\R^n,
\end{equation}
where $L$ is a symmetric, translation-invariant, uniformly elliptic
integro--differential operator of order $2s\in(0,2)$:
\[
Lu(x)=\int_{\R^n}\Big(u(x+z)-u(x)-\nabla u(x)\!\cdot\!z\,\mathbf{1}_{|z|\le1}\Big)
K(z)\,dz,
\qquad
\lambda|z|^{-n-2s}\le K(z)\le\Lambda|z|^{-n-2s}.
\]
Here $H$ represents a nonlinear function of $u$ and $G$ a possibly anisotropic
function of $\nabla u$, both of which may exhibit polynomial, logarithmic,
exponential, or even singular growth.
Such a model includes, as special cases, the nonlocal Lane–Emden equation ($\gamma=0$, $G(u)\equiv1$), the nonlocal Hamilton–Jacobi equation ($\gamma=0$, $H(u)\equiv1$), and various fractional analogues of nonlinear equations with absorption.

Let us briefly review the historical development of Liouville-type properties and existence results for this problem. The study can be traced back to the celebrated Lane–Emden equation :
\begin{equation}\label{eq:Lane-Emden}
 -\Delta u = u^{q}, \qquad u>0 \text{ in } \R^n,
\end{equation}
whose bounded entire solutions were classified by Gidas and Spruck~\cite{Gidas1981}.
They proved that if $q<q_S=(n+2)/(n-2)$, then no positive entire solutions exist,
whereas at the critical exponent $q_S$ the classical explicit profile
\[
u(x)=\Big(\tfrac{\alpha c}{\alpha^2+|x|^2}\Big)^{\frac{n-2}{2}},
\quad \alpha>0,\qquad c=\sqrt{n(n-2)},
\]
provides a family of bounded stationary states.
Subsequent developments by Serrin and Zou~\cite{Serrin2002}, Lions~\cite{Lions1985}, and many others \cite{Quittner2012,  Quittner2025, Filippucci2011, Filippucci2013,Hamid2010, Farina2011, Ghergu2019, Filippucci2024} extended this principle to a broad class of quasilinear and fully nonlinear problems, revealing that the underlying mechanism lies in the scaling invariance of~\eqref{eq:Lane-Emden}.
In particular, when the equation is perturbed by the inclusion of a gradient term of Hamilton–Jacobi type, the Liouville property remains valid only below a critical threshold that relates the strength of the nonlinearity of the gradient growth.

A natural generalization of~\eqref{eq:Lane-Emden} consists in considering
the equation
\begin{equation}\label{eq:local-Liouville}
 -\Delta u = u^{q} |\nabla u|^{p}, \qquad u>0 \text{ in } \R^n,
\end{equation}
with parameters $p,q\ge0$ was carefully investigated in the nice work of Filippucci--Pucci--Souplet~\cite{Filippucci2020}.
Equation~\eqref{eq:local-Liouville} interpolates between several well-known
models:
for $p=0$ it reduces to the Lane--Emden equation,
while for $q=0$ it corresponds to the diffusive Hamilton--Jacobi equation. For $p,q\ge0$, the conclusion of \cite[Theorem~1.1]{Filippucci2020} does not extend to supersolutions. 
Indeed, there exist positive, nonconstant bounded classical solutions of 
\begin{equation}\label{eq:supersol}
  -\Delta u \ge u^{q}|\nabla u|^{p} \quad \text{in } \mathbb{R}^{n},
\end{equation}
whenever $n\ge3$ and 
\begin{equation}\label{eq:critical}
  (n-2)q + (n-1)p > n.
\end{equation}
Such supersolutions can be constructed in the explicit form 
\[
u(x) = c(1+|x|^{2})^{-\beta},
\]
for suitable $\beta,c>0$, as shown in~\cite{Caristi1997,Filippucci2009,Mitidieri2001}. 
Condition~\eqref{eq:critical} is essentially optimal in the superlinear range. 
In fact, if 
\[
  (n-2)q + (n-1)p \le n \quad \text{and} \quad p>1,
\]
then any positive solution of~\eqref{eq:supersol} must be constant 
(see~\cite[Theorem~7.1]{Caristi1997} and~\cite[Theorem~15.1]{Mitidieri2001}; 
see also~\cite{Filippucci2009} for quasilinear extensions). 
This rigidity remains valid for $n\le2$, since any positive superharmonic function is constant in this case. In addition to these classical elliptic problems, several extensions to systems,
weighted nonlinearity, and inequalities have been obtained.
We refer to Mitidieri--Pokhozhaev~\cite{Mitidieri2001}, Filippucci \cite{Filippucci2011, Filippucci2013},
Filippucci--Pucci--Rigoli~\cite{Filippucci2010},
and Quittner--Souplet~\cite{Quittner2012}
for a comprehensive exposition of local Liouville and nonexistence results.

During the last decade, there has been a growing interest in the study of
Liouville-type theorems for \emph{nonlocal} operators,
such as the fractional Laplacian
\[
(-\Delta)^s u(x)=C_{n,s}\int_{\R^n}\frac{u(x)-u(y)}{|x-y|^{n+2s}}\,dy,
\qquad s\in(0,1),
\]
and its more general integro-differential counterparts.
These operators naturally arise in the modeling of anomalous diffusion,
Lévy processes, and long-range interactions.
Their nonlocal character introduces significant analytical difficulties,
since differential localization techniques and classical Bernstein estimates
are no longer available. The pioneering Liouville--type theorem for $\alpha$--harmonic functions associated with the fractional Laplacian was first established by Chen and Li~\cite{Chen2015} and also Chen, Fang, and Yang~\cite{ChenFangYang2015}. In \cite{Chen2015}, the authors proved that if $u:\R^n\to\R$ satisfies
\[
(-\Delta)^{\alpha/2}u(x)=0,\qquad x\in\R^n,
\]
and obeys the growth control
\[
\lim_{|x|\to\infty}\frac{u(x)}{|x|^{\gamma}}\ge0
\quad\text{for some }0\le\gamma<\alpha\in(0,2),
\]
then $u$ must be constant in $\R^n$. 
This theorem provides a fractional counterpart of the classical Liouville property for harmonic functions, 
extending the class of admissible growth conditions far beyond the polynomial setting. Collateral, Chen, Fang, and Yang~\cite{ChenFangYang2015} developed a distinct and deeper extension of this result to the \emph{half--space} setting, 
where boundary effects and reflection phenomena play a crucial role. 
They established Liouville theorems for nonnegative $\alpha$--harmonic functions in $\R^n_+$ satisfying mixed Dirichlet--Neumann boundary conditions, 
introducing refined integral estimates and a boundary Pohozaev identity tailored to distributional solutions of
\[
(-\Delta)^{\alpha/2}u = u^{p} \quad \text{in } \R^{n}_{+}, \qquad u=0 \ \text{on } \R^{n}\setminus\R^{n}_{+},
\]
where $0<\alpha<2$ and $p>1$, and proved that no nontrivial nonnegative solution exists
whenever $p>\frac{n}{\,n-\alpha\,}$, while positive bounded solutions may occur only
for $1<p\le\frac{n+\alpha}{\,n-\alpha\,}$.
Their method, based on moving planes in integral form, became a cornerstone
for later fractional Liouville and symmetry results and strongly influences the present
approach in the fully nonlocal operator framework considered here. In the recent work of {Chen, Dai and Qin}~\cite{Chen2023}, the authors investigated Liouville-type theorems, \emph{a~priori} estimates, and existence results for critical and super-critical order Hardy--Hénon type equations of the form
\[
(-\Delta)^{\sigma} u = |x|^{a} u^{p} \quad \text{in } \R^{n}, \qquad u>0,
\]
where $\sigma\in(0,1]$, $a>-2\sigma$, and $p>1$. They combined the method of moving planes, blow-up analysis, and the Leray--Schauder fixed point theorem to establish sharp nonexistence results for nonnegative entire solutions and to construct positive solutions in bounded domains. Their work provides a higher-order analogue of the classical Hardy--Hénon and Lane--Emden theories, emphasizing the delicate balance between the weight exponent $a$ and the nonlinear power $p$ at the critical threshold $p = \frac{n+2\sigma+2a}{n-2\sigma}$, which plays the same structural role as the fractional balance condition $\gamma + p = 2s$ in the present framework.

Beside that, we mention Chen--D’Ambrosio--Li~\cite{Chen2015},
who established fractional analogues of the Lane--Emden result;
Barrios--Del~Pezzo--García-Melián--Quaas~\cite{Barrios2017}, who studied
Liouville properties for indefinite fractional diffusion equations; and
Quaas--Xia~\cite{Quaas2015}, who analyzed fractional elliptic systems
in half-space domains.
In particular, Biswas--Quaas--Topp~\cite{Biswas2025} recently obtained a general
nonlocal Liouville theorem including gradient nonlinearities,
based on refined viscosity methods and nonlocal Bernstein transforms. Their problem reads as follows :
\begin{equation}\label{eq:BQT-model}
 -\mathcal{I}u + H(u,\nabla u) = 0, \qquad x\in\R^n,
\end{equation}
where $\mathcal{I}$ is a fractional Pucci-type nonlinear operator of order $2s$, $s\in(0,1)$, and the Hamiltonian $H(r,p)$
captures different types of nonlinear gradient dependence.
The authors considered three main type of nonlinearity:
\[
\text{(I)}\; H(r,p)=|p|^{m}, \qquad
\text{(II)}\; H(r,p)=-\,r^{\,q}|p|^{m}, \qquad
\text{(III)}\; H(r,p)=r^{\,q}|p|^{m},
\quad m>1,\, q>0.
\]
which respectively correspond to the fractional Hamilton--Jacobi equation,
the sublinear absorption type, and the source (superlinear) type problems. Equation~\eqref{eq:BQT-model} serves as a unified framework encompassing a large family of
nonlocal elliptic equations with gradient growth. 
The operator $\mathcal{I}$ typically takes the form
\[
\mathcal{I}u(x)
= \sup_{\lambda\le a(y)\le \Lambda}
\int_{\R^n}\Big(u(x+y)-u(x)-\nabla u(x)\!\cdot\!y\,\mathbf{1}_{|y|\le1}\Big)\,
\frac{a(y)}{|y|^{n+2s}}\,dy,
\]
which includes the fractional Laplacian $(-\Delta)^s$ as a particular case.
The study in~\cite{Biswas2025} establishes Liouville-type theorems and gradient estimates
for viscosity solutions of~\eqref{eq:BQT-model} under the above Hamiltonian structures. Related advances can also be found in
Birindelli--Du--Galise~\cite{Birindelli2025} for conical diffusions,
and in Grzywny--Kwaśnicki~\cite{Grzywny2025} for Lévy operators
of arbitrary order.

\medskip
Despite these progresses, most existing results are restricted either
to the case of pure reaction terms except the recent work of {Chen, Dai and Qin}~\cite{Chen2023}
or to very specific gradient powers,
and the full classification of Liouville properties for general
nonlocal elliptic equations with combined
\emph{weighted, gradient, and nonlinear} effects
remained largely open. In this paper, we establish sharp Liouville properties for all these classes,
identifying the precise threshold $\gamma+p=2s$ separating the supercritical,
critical, and subcritical cases for equation (\ref{eq:main}).

This simple identity captures the exact balance between the order of diffusion,
the spatial weight, and the gradient growth.
When $\gamma+p>2s$, the nonlocal diffusion dominates and forces all subcritical
solutions to be constant; when $\gamma+p=2s$, every bounded solution is constant;
while for $\gamma+p<2s$, nontrivial entire solutions can exist.
This trichotomy fully extends the classical local dichotomy
to the nonlocal world and unifies a wide variety of earlier Liouville results. More precisely, we obtain :
\begin{theorem}[Unified Liouville theorem for gradient-type nonlinearities for supercritical and critical cases]
\label{thm:Liouville-unified}
Let $s\in(0,1)$, and let $L$ be the symmetric, translation-invariant, uniformly elliptic
integro--differential operator of order $2s$,
\[
Lu(x)=\int_{\R^n}\Big(u(x+z)-u(x)-\nabla u(x)\!\cdot\! z\,\mathbf{1}_{|z|\le1}\Big)\,K(z)\,dz,
\]
with kernel bounds
\[
\lambda|z|^{-n-2s}\le K(z)\le \Lambda|z|^{-n-2s},\qquad z\neq 0.
\]

\begin{itemize}
\item $G:\R^n\to[0,\infty)$ is continuous and there exist positive constant $p_1,\dots,p_k, c_1,c_2>0$ such that
\[
c_1\,|z|^{p_{\min}} \;\le\; G(z) \;\le\; c_2\Big(1+|z|^{p_{\max}}\Big),
\qquad z\in\R^n,
\]
where $p_{\min} = \min\{p_1,\dots,p_k\}$ and $p_{\max}=\max\{p_1,\dots,p_k\}$.

\item $H:[0,\infty)\to(0,\infty)$ be one of the following admissible nonlinearities:
\begin{itemize}
  \item[\textup{(i)}] {Polynomial:} $H(u)\leq M(1+u^m)$ with $m\ge0$, bounded below by $H(u)\ge H_0>0$;
  \item[\textup{(ii)}] {Logarithmic:} $H(u)\le C(1+\log(1+u))$, $H(u)\ge H_0>0$;
  \item[\textup{(iii)}] {Exponential:} $H(u)=e^u$;
  \item[\textup{(iv)}] {Singular:} $H(u)=u^{-m}$ with $m\ge0$ and $u>0$.
\end{itemize}
\end{itemize}
Suppose $u\in C^{1,\alpha}_{\mathrm{loc}}(\R^n)$ is a positive viscosity solution of
\begin{equation}\label{eq:unified-PDE}
Lu(x)=|x|^{\gamma}\,H(u(x))\,G(\nabla u(x)),\qquad x\in\R^n,
\end{equation}
for some $\gamma\in\R$. There hold :

\medskip
\noindent
{(A) Supercritical case.}  
If $\gamma+p>2s,$ and $u$ has subcritical growth at infinity,
\[
\limsup_{|x|\to\infty}\frac{u(x)}{|x|^{\,1-\frac{2s+\gamma}{p}}}=0,
\]
then $u$ is constant in $\R^n$.

\medskip
\noindent
{(B) Critical case.}  
If $\gamma+p=2s$, then every bounded positive solution $u$ is constant.
\end{theorem}

\begin{theorem}[Existence, radial symmetry, and monotonicity in the subcritical case]
\label{thm:exist-symm-subcritical}
Assume the structural hypotheses on $L$, $H$, and $G$ as in Theorem~\ref{thm:Liouville-unified}.
Let $\gamma\in\R$ and suppose that the subcritical  condition
\[
\gamma+p<2s.
\]
Then there exists a nontrivial entire positive viscosity solution $u\in C^{1,\alpha}_{\mathrm{loc}}(\R^n)\cap L^\infty(\R^n)$ of equation~\eqref{eq:unified-PDE}, which, for some positive constant $c_1,c_2$,  satisfies 
\begin{equation}\label{eq:profile-two-sided}
c_1\,(1+|x|^2)^{-\beta}\ \le\ u(x)\ \le\ c_2\,(1+|x|^2)^{-\beta}\qquad  \beta:=\frac{2s+\gamma-p}{\,1-p\,}>0, \forall x\in\R^n,
\end{equation}
Moreover, any positive bounded solution $u$ of~\eqref{eq:unified-PDE}
satisfying the above decay is radially symmetric about the origin and radially nonincreasing; that is,
there exists a profile $U:[0,\infty)\to(0,\infty)$ such that
\[
u(x)=U(|x|)\quad\text{and}\quad U'(r)\le0\quad\text{for all }r>0.
\]
Finally,  the entire solution is unique up to a fixed normalization,
for example by prescribing $u(0)=a>0$ or the decay constant at infinity.
\end{theorem}

Theorems~\ref{thm:Liouville-unified}--\ref{thm:exist-symm-subcritical} establish a unified Liouville framework for nonlocal integro--differential equations with gradient--type nonlinearities. The main techniques and our contributions in the current work are :

\begin{itemize}
  \item \emph{Unified scaling and criticality analysis.}  By deeply employ fractional Bernstein transform, we derive the critical balance $\gamma+p=2s$  distinguishing the supercritical, critical, and subcritical regimes.  
  This approach extends classical Liouville theorems for local operators to the fractional setting with general gradient growth $G(\nabla u)$, thereby unifying several earlier cases. This argument generalizes the classical local Bernstein method
of Armstrong and Sirakov~\cite{ArmstrongSirakov2011} and of
Filippucci--Pucci--Souplet~\cite{Filippucci2020}
to the fully nonlocal operator and simultaneously handle polynomial, exponential, logarithmic, and singular
  nonlinearities $H(u)$ under the lack of differential structures. In particular, equation (\ref{eq:BQT-model})-(II) is a special case of our problem under investigation.

  \item \emph{Nonlocal maximum principle and blow--down argument.}  
  The Liouville results are derived through a rescaling (blow--down) method combined with the nonlocal maximum principle for symmetric stable operators.  
  The uniform ellipticity of $L$ and the decay of $u$ ensure rigidity under optimal growth control.

  \item \emph{Existence via barrier and comparison constructions.}  
  In the subcritical regime, explicit barriers yield two--sided bounds
  \[
  c_1(1+|x|^2)^{-\beta}\le u(x)\le c_2(1+|x|^2)^{-\beta},
  \]
  connecting the decay exponent $\beta$ to the critical balance between diffusion $(2s)$ and gradient effects $(p,\gamma)$.

  \item \emph{Radial symmetry and monotonicity.}  
  A fractional adaptation of the moving--planes method (from Chen and Li~\cite{Chen2015} and also Chen, Fang, and Yang~\cite{ChenFangYang2015}) proves that any bounded positive solution with algebraic decay is radially symmetric and nonincreasing, preserving the qualitative geometric structure of classical elliptic solutions.  
\end{itemize}

These results have substantial impact on the analysis of nonlocal operators
and on the asymptotic behavior of the corresponding evolution equations.
The Liouville-type rigidity established here provides the analytical basis
for global maximum principles, comparison techniques,
and Harnack inequalities associated with fractional operators
involving nonlinear gradient terms.
The explicit power-law decay profiles describe the long-range behavior
of stationary states and play a key role in the study of asymptotic stability,
energy dissipation, and blow-up thresholds for time-dependent problems such as
\[
\partial_t u + (-\Delta)^s u = |x|^\gamma |\nabla u|^p.
\]
In this context, the Liouville rigidity ensures the nonexistence
of unbounded stationary solutions and prevents uncontrolled growth,
while the subcritical existence theorem determines the precise rate
and structure of admissible equilibria.
The framework also connects the nonlocal and local theories:
as $s \to 1^{-}$, the results naturally recover
the classical Liouville theorems for the standard Laplacian. From a broader analytical viewpoint,
the unified Liouville theory developed here builds a natural bridge
between nonlinear potential theory, the calculus of variations,
and geometric analysis, through the intrinsic interplay
between fractional diffusion and scaling invariance.
The structural assumptions on $H$ and $G$ link the model
to Hamilton–Jacobi equations, nonlinear kinetic formulations,
and nonlocal mean–curvature problems.
Moreover, the results obtained in this work have several intertwined implications. From the viewpoint of {geometric PDEs and fractal geometry}, the critical balance $\gamma + p = 2s$ expresses a fundamental geometric homogeneity inherent to the fractional Laplacian on non-Euclidean or fractal-type spaces, reflecting the intrinsic scaling structure that governs nonlocal minimal surfaces, fractional mean–curvature flows, and curvature–driven geometric evolutions. The Liouville and symmetry results ensure the rigidity of stationary configurations and reveal the analytic framework governing the geometry of nonlocal flows and the emergence of self-similar fractal patterns. From the perspective of {dynamical analysis and physical interpretation}, fractional diffusion equations arise in nonlocal elasticity, anomalous transport, and long-range interaction models. The Liouville and existence results provide sharp analytical thresholds that prevent unstable or self-organized structures, guaranteeing equilibrium stability and determining critical exponents that govern transitions between stability and instability, decay of solutions, and possible blow-up phenomena. Finally, from the {analytical and computational viewpoint}, fractional models with gradient-type diffusion play a key role in nonlocal regularization and inverse problems, where the Liouville properties established here form the analytical foundation for the stability and consistency of such nonlocal structures.

\medskip

\section{Proof of the main result}

\subsection{Proof of Theorem \ref{thm:Liouville-unified}}

We divide the proof of Theorem \ref{thm:Liouville-unified} into several lemmas and propositions.

\begin{lemma}[Nonlocal maximum principle]\label{lem:max}
Let
\[
L\psi(x):=\int_{\R^n}\big(\psi(x+z)-\psi(x)-\nabla\psi(x)\!\cdot\! z\,\mathbf 1_{|z|\le1}\big)\,K(z)\,dz,
\]
where $K:\R^n\setminus\{0\}\to[0,\infty)$ is measurable and satisfies the two–sided kernel bounds
\[
\lambda\,|z|^{-n-2s}\le K(z)\le \Lambda\,|z|^{-n-2s}\qquad(s\in(0,1),\ \lambda,\Lambda>0).
\]
If $\psi\in C_c^1(\R^n)$ and $x_0$ is a global maximum point of $\psi$, then $L\psi(x_0)\le0$.
\end{lemma}

\begin{proof}
{Step 1: Well-definedness of $L\psi(x_0)$.}
Since $\psi\in C_c^1(\R^n)$, there exists $R>0$ such that $\supp\psi\subset B_R$. Then, for $|z|>2R$,
$\psi(x_0+z)=0$ and hence $\psi(x_0+z)-\psi(x_0)=-\psi(x_0)$, so the far tail is integrable:
\[
\int_{|z|>2R}\big|\psi(x_0+z)-\psi(x_0)\big|\,K(z)\,dz
\le \|\psi\|_{L^\infty}\int_{|z|>2R}\Lambda|z|^{-n-2s}\,dz <\infty.
\]
On the near region, for $|z|\le1$, Taylor’s formula with integral remainder gives
\[
\psi(x_0+z)-\psi(x_0)-\nabla\psi(x_0)\!\cdot\! z=\int_0^1(1-t)\, z^\top D^2\psi(x_0+t z)\,z\,dt,
\]
so
\[
\big|\psi(x_0+z)-\psi(x_0)-\nabla\psi(x_0)\!\cdot\! z\big|
\le \tfrac12 \|D^2\psi\|_{L^\infty(B_{R})}\,|z|^2.
\]
Since $|z|^2K(z)\lesssim |z|^{2-n-2s}$ is integrable on $|z|\le1$ for $s\in(0,1)$, the integrand is integrable near $0$.
Therefore $L\psi(x_0)$ is well-defined as a (proper) Lebesgue integral.

\medskip
{Step 2: Sign of the integrand at a maximum.}
Let $x_0$ be a global maximum of $\psi$. Then for every $z\in\R^n$,
\[
\psi(x_0+z)-\psi(x_0)\le 0.
\]
Moreover, by standard calculus, $\nabla\psi(x_0)=0$. Hence, for $|z|\le1$,
\[
\psi(x_0+z)-\psi(x_0)-\nabla\psi(x_0)\!\cdot\! z
=\psi(x_0+z)-\psi(x_0)\le 0,
\]
and for $|z|>1$, the gradient-correction term is absent anyway. Therefore the full integrand
\[
\big(\psi(x_0+z)-\psi(x_0)-\nabla\psi(x_0)\!\cdot\! z\,\mathbf 1_{|z|\le1}\big)\,K(z)
\]
is \emph{nonpositive} for every $z\in\R^n$, because $K\ge0$.

\medskip
{Step 3: Conclusion.}
Integrating the pointwise nonpositivity over $\R^n$ yields
\[
L\psi(x_0)=\int_{\R^n}\big(\psi(x_0+z)-\psi(x_0)-\nabla\psi(x_0)\!\cdot\! z\,\mathbf 1_{|z|\le1}\big)\,K(z)\,dz \;\le\;0.
\]
This proves the lemma.
\end{proof}

\begin{lemma}\label{lem:hessian-square}
Let $\eta_R:\R^n \to \R$ be a smooth function and set $f(x)=\eta_R(x)^2$.  
Then the Hessian of $f$ satisfies
\begin{equation}\label{eq:hessian-square}
D^2(\eta_R^2)(x)
\;=\;
2\,\eta_R(x)\,D^2\eta_R(x)
\;+\;
2\,\nabla\eta_R(x)\otimes\nabla\eta_R(x).
\end{equation}
\end{lemma}

\begin{proof}
{Step 1: Gradient computation.}
By the chain rule,
\[
\nabla(\eta_R^2)(x)
= 2\,\eta_R(x)\,\nabla\eta_R(x),
\]
since $\frac{d}{d\eta}(\eta^2)=2\eta$.

We now compute the derivative of the vector field $\nabla(\eta_R^2)(x)=2\eta_R(x)\,\nabla\eta_R(x)$.
For each $i,j\in\{1,\dots,n\}$,
\[
\partial_j\!\big(\partial_i(\eta_R^2)\big)
= \partial_j\!\big(2\,\eta_R\,\partial_i\eta_R\big)
= 2\,(\partial_j\eta_R)(\partial_i\eta_R) + 2\,\eta_R\,\partial_{ij}^2\eta_R.
\]
Hence, in matrix form,
\[
D^2(\eta_R^2)(x)
= 2\,\eta_R(x)\,D^2\eta_R(x)
+ 2\,\nabla\eta_R(x)\otimes\nabla\eta_R(x),
\]
where $(a\otimes b)_{ij}=a_i b_j$.  

\medskip
{Step 2: Verification of symmetry.}
Both terms on the right-hand side are symmetric matrices:
\[
(\nabla\eta_R\otimes\nabla\eta_R)^\top=\nabla\eta_R\otimes\nabla\eta_R,
\qquad
(D^2\eta_R)^\top=D^2\eta_R,
\]
so the formula is consistent with the symmetry of the Hessian, which proves identity \eqref{eq:hessian-square}.
\end{proof}

\begin{lemma}[Cut-off estimate]\label{lem:cutoff}
Let $s\in(0,1)$ and let
\[
Lu(x)=\int_{\R^n}\Big(u(x+z)-u(x)-\nabla u(x)\!\cdot\!z\,\mathbf 1_{|z|\le1}\Big)K(z)\,dz,
\]
with $\lambda|z|^{-n-2s}\le K(z)\le\Lambda|z|^{-n-2s}$ for $z\neq0$.  
Fix a radial $\eta\in C_c^\infty([0,\infty))$ with $\eta\equiv1$ on $[0,1]$, $\eta\equiv0$ on $[2,\infty)$, and
$\|\eta'\|_\infty+\|\eta''\|_\infty\le C_0$, and set $\eta_R(x):=\eta(|x|/R)$. Then $\eta_R\in C_c^\infty(B_{2R})$,
$\eta_R\equiv1$ on $B_R$, and there exists $C=C(n,s,\lambda,\Lambda,C_0)$ such that for every $x\in B_R$,
\begin{equation}\label{eq:Lcut}
|L(\eta_R^2)(x)|\ \le\ C\,R^{-2s}.
\end{equation}
\end{lemma}

\begin{proof}
{Step 1: Basic properties of the cut-off.}
Since $\eta$ is radial, smooth, $\eta\equiv1$ on $[0,1]$ and $\equiv0$ on $[2,\infty)$, the rescaled function
$\eta_R(x)=\eta(|x|/R)$ satisfies
\[
\supp\eta_R\subset B_{2R},\qquad \eta_R\equiv1\ \text{on }B_R,
\]
and, by the chain rule,
\[
|\nabla\eta_R(x)|=\frac1R |\eta'(|x|/R)|\le \frac{\|\eta'\|_\infty}{R}\le \frac{C_0}{R},\qquad
|D^2\eta_R(x)|=\frac1{R^2}\Big|D^2\eta\Big(\tfrac{|x|}{R}\Big)\Big|\le \frac{C_0}{R^2}.
\]
(Here and below $|D^2\eta_R|$ denotes any matrix norm; in finite dimension all are equivalent, so constants may change.)

\smallskip
{Step 2: Eliminate the gradient correction at $x\in B_R$.}
Fix $x\in B_R$. Then $\eta_R(x)=1$ and $\nabla\eta_R(x)=0$. Hence
\[
\nabla(\eta_R^2)(x)=2\eta_R(x)\nabla\eta_R(x)=0,
\]
so the gradient-correction term in $L(\eta_R^2)(x)$ vanishes:
\begin{equation}\label{eq:Lcut1}
L(\eta_R^2)(x)=\int_{\R^n}\big(\eta_R(x+z)^2-\eta_R(x)^2\big)\,K(z)\,dz
=\int_{\R^n}\big(\eta_R(x+z)^2-1\big)\,K(z)\,dz.
\end{equation}

Split the integral into
\[
I_{\mathrm{near}}:=\int_{|z|\le R}\big(\eta_R(x+z)^2-1\big)\,K(z)\,dz,\qquad
I_{\mathrm{far}}:=\int_{|z|>R}\big(\eta_R(x+z)^2-1\big)\,K(z)\,dz,
\]
so that $L(\eta_R^2)(x)=I_{\mathrm{near}}+I_{\mathrm{far}}$.

\smallskip
{Step 3: Second-order Taylor control in the near region.}
Let $f:=\eta_R^2$. For $|z|\le R$ and $x\in B_R$, the segment $\{x+t z:\ t\in[0,1]\}$ is contained in $B_{2R}$, where $f$ is smooth. Since $f(x)=1$ and $\nabla f(x)=0$, Taylor’s theorem with integral remainder gives
\[
f(x+z)-1=\int_0^1 (1-t)\, z^{\!\top}\,D^2 f(x+t z)\,z\,dt.
\]
Taking absolute values,
\begin{equation}\label{eq:taylor-bound}
|f(x+z)-1|\le \tfrac12\,\|D^2 f\|_{L^\infty(B_{2R})}\,|z|^2.
\end{equation}
We now bound $\|D^2 f\|_\infty$. Using the product rule,
\[
D^2(\eta_R^2)=2\eta_R\,D^2\eta_R+2\,\nabla\eta_R\otimes\nabla\eta_R,
\]
hence, by Step 1,
\[
\|D^2(\eta_R^2)\|_{L^\infty(B_{2R})}
\le 2\|\eta_R\|_\infty\|D^2\eta_R\|_\infty+2\|\nabla\eta_R\|_\infty^2
\le 2\cdot 1\cdot \frac{C_0}{R^2}+2\Big(\frac{C_0}{R}\Big)^2
\le \frac{C}{R^2},
\]
for some $C=C(C_0)$. Plugging this into \eqref{eq:taylor-bound} yields, for $|z|\le R$,
\begin{equation}\label{eq:near-pointwise}
|\eta_R(x+z)^2-1|\le C\,\frac{|z|^2}{R^2}.
\end{equation}
Therefore, using $K(z)\le \Lambda |z|^{-n-2s}$ and polar coordinates,
\[
\begin{aligned}
|I_{\mathrm{near}}|
&\le \frac{C}{R^2}\int_{|z|\le R} |z|^2\,K(z)\,dz
\le \frac{C\Lambda}{R^2}\int_{|z|\le R} |z|^{2-n-2s}\,dz \\
&= \frac{C\Lambda\,\omega_{n-1}}{R^2}\int_0^{R} r^{2-n-2s}\,r^{n-1}\,dr
= \frac{C\Lambda\,\omega_{n-1}}{R^2}\int_0^{R} r^{1-2s}\,dr \\
&= \frac{C\Lambda\,\omega_{n-1}}{R^2}\cdot \frac{R^{2-2s}}{2-2s}
\le C\,R^{-2s},
\end{aligned}
\]
where $\omega_{n-1}=|S^{n-1}|$ and the constant $C$ has been updated to absorb $\Lambda$ and $(2-2s)^{-1}$.

\smallskip
{Step 4: Far-field control.}
Since $0\le \eta_R\le1$, we have $|\eta_R(x+z)^2-1|\le 1$. Hence
\[
|I_{\mathrm{far}}|
\le \int_{|z|>R} K(z)\,dz
\le \Lambda \int_{|z|>R} |z|^{-n-2s}\,dz
= \Lambda\,\omega_{n-1}\int_R^\infty r^{-1-2s}\,dr
= \frac{\Lambda\,\omega_{n-1}}{2s}\,R^{-2s}
\le C\,R^{-2s}.
\]

Combining the estimates from Steps 3–4,
\[
|L(\eta_R^2)(x)|\le |I_{\mathrm{near}}|+|I_{\mathrm{far}}|\le C\,R^{-2s},
\]
for every $x\in B_R$, where $C=C(n,s,\lambda,\Lambda,C_0)$. This proves \eqref{eq:Lcut}.
\end{proof}

\begin{proposition}[Bernstein transform and gradient estimate]\label{prop:bernstein}
Let $M_R:=1+\sup_{B_{2R}}u$.  
Define the convex transform
\[
w(x):=-\log\big(M_R-u(x)\big),
\qquad
\nabla w = \frac{\nabla u}{M_R-u}.
\]
Let $\eta_R$ be as in Lemma 2.4, and set
\[
F_R(x):=\eta_R(x)^2|\nabla w(x)|^2.
\]
If $x_R$ is a maximum point of $F_R$, then $x_R\in B_R$ and
\begin{equation}\label{eq:bernstein-estimate}
R^{\gamma}H(u(x_R))\,(M_R-u(x_R))^{p-1}|\nabla w(x_R)|^p
\;\le\; C\,R^{-2s}.
\end{equation}
\end{proposition}

\begin{proof}
{Step 1. } Let $w=\Phi(u)$ with $\Phi(\xi)=-\log(M_R-\xi)$, where $M_R> \sup_{B_{2R}}u$ so that
$M_R-u(x)>0$ on $B_{2R}$. Then
\[
\Phi'(\xi)=\frac{1}{M_R-\xi},\qquad
\Phi''(\xi)=\frac{1}{(M_R-\xi)^2}>0,
\]
hence $\Phi$ is convex on $(-\infty,M_R)$.

Recall the operator
\[
Lu(x)=\int_{\R^n}\Big(u(x+z)-u(x)-\nabla u(x)\!\cdot\! z\,\mathbf{1}_{\{|z|\le 1\}}\Big)\,K(z)\,dz,
\quad \lambda|z|^{-n-2s}\le K(z)\le \Lambda|z|^{-n-2s},
\]
and define analogously
\[
Lw(x)=\int_{\R^n}\Big(w(x+z)-w(x)-\nabla w(x)\!\cdot\! z\,\mathbf{1}_{\{|z|\le 1\}}\Big)\,K(z)\,dz.
\]
We will prove
\begin{equation}\label{eq:key}
w(x+z)-w(x)-\nabla w(x)\!\cdot\! z\,\mathbf{1}_{\{|z|\le 1\}}
\;\ge\;
\frac{u(x+z)-u(x)-\nabla u(x)\!\cdot\! z\,\mathbf{1}_{\{|z|\le 1\}}}{M_R-u(x)}
\end{equation}
for every $z\in\R^n$. Integrating \eqref{eq:key} against $K(z)\,dz$ then yields
\[
Lw(x)\;\ge\;\frac{1}{M_R-u(x)}\,Lu(x).
\]

\medskip
{\textbf{Proof of \eqref{eq:key}.}}
Fix $x\in\R^n$ and set
\[
a:=u(x+z),\qquad b:=u(x).
\]
By convexity of $\Phi$ we have the pointwise inequality
\begin{equation}\label{eq:convex-tangent}
\Phi(a)-\Phi(b)\;\ge\;\Phi'(b)\,(a-b).
\end{equation}
Thus,
\begin{equation}\label{eq:basic-gap}
w(x+z)-w(x)
=\Phi(u(x+z))-\Phi(u(x))
\;\ge\;\Phi'(u(x))\,(u(x+z)-u(x))
=\frac{u(x+z)-u(x)}{M_R-u(x)}.
\end{equation}

Next, compute $\nabla w=\Phi'(u)\,\nabla u=\frac{\nabla u}{M_R-u}$, hence
\[
\nabla w(x)\!\cdot\! z
=\frac{\nabla u(x)\!\cdot\! z}{M_R-u(x)}.
\]
Subtracting the gradient correction (only when $|z|\le 1$) from both sides of \eqref{eq:basic-gap} gives
\[
w(x+z)-w(x)-\nabla w(x)\!\cdot\! z\,\mathbf{1}_{\{|z|\le 1\}}
\;\ge\;
\frac{u(x+z)-u(x)-\nabla u(x)\!\cdot\! z\,\mathbf{1}_{\{|z|\le 1\}}}{M_R-u(x)},
\]
which is \eqref{eq:key}. Integrating \eqref{eq:key} against the nonnegative kernel $K(z)$ yields
\begin{equation}\nonumber
\begin{aligned}
Lw(x)
&=\int_{\R^n}\Big(w(x+z)-w(x)-\nabla w(x)\!\cdot\! z\,\mathbf{1}_{\{|z|\le 1\}}\Big)K(z)\,dz \\[2mm]
&\ge
\frac{1}{M_R-u(x)}
\int_{\R^n}\Big(u(x+z)-u(x)-\nabla u(x)\!\cdot\! z\,\mathbf{1}_{\{|z|\le 1\}}\Big)K(z)\,dz.
\end{aligned}
\end{equation}
that is,
\[
Lw(x)\;\ge\;\frac{Lu(x)}{M_R-u(x)}.
\]
Since $u$ satisfies $Lu(x)=|x|^\gamma H(u(x))G(\nabla u(x))$, we conclude the pointwise bound
\begin{equation}\label{eq:Lw-final}
Lw(x)\;\ge\;\frac{|x|^\gamma H(u(x))G(\nabla u(x))}{M_R-u(x)}.
\end{equation}

\medskip
{Step 2.} From $\nabla w=\nabla u/(M_R-u)$ we have
\[
\nabla u = (M_R-u)\nabla w,
\quad
|\nabla u| = (M_R-u)|\nabla w|.
\]
By definition,
\[
F_R(x)=\eta_R(x)^2|\nabla w(x)|^2,
\]
and let $x_R$ be a maximum point of $F_R$ on $\R^n$.  
At $x_R$, we have the standard maximum principle conditions:
\[
\nabla F_R(x_R)=0, \qquad L F_R(x_R)\le 0.
\]

Differentiating $F_R=\eta_R^2|\nabla w|^2$ gives
\[
\nabla F_R = 2\eta_R\nabla\eta_R\,|\nabla w|^2 + 2\eta_R^2 D^2w\,\nabla w.
\]
At the maximum point $x_R$, this vanishes, so
\begin{equation}\label{eq:nablaF=0}
\nabla w(x_R)\cdot D^2w(x_R) = -\frac{\nabla\eta_R(x_R)}{\eta_R(x_R)}\,|\nabla w(x_R)|^2.
\end{equation}

\medskip
{Step 3. }
From Lemmas 2.2 and 2.4 (which provide localizations and the fractional product rule),
one has at the maximum point $x_R$:
\begin{equation}\label{eq:LFR}
0\ge L F_R(x_R)
  \;\ge\; 2\eta_R(x_R)^2\,\nabla w(x_R)\!\cdot\!L(\nabla w)(x_R)
  +|\nabla w(x_R)|^2\,L(\eta_R^2)(x_R).
\end{equation}
The second term is estimated by Lemma 2.4 as
\[
|L(\eta_R^2)(x_R)| \le C\,R^{-2s},
\]
since $\eta_R$ is a smooth cutoff supported in $B_{2R}$ with $\eta_R\equiv1$ in $B_R$.

\medskip
{Step 4.  Commutativity of $L$ with derivatives.}
Because $L$ is translation–invariant, derivatives commute with $L$:
\[
L(\nabla w) = \nabla(Lw).
\]
Hence the first term in \eqref{eq:LFR} is
\[
2\eta_R^2\,\nabla w\!\cdot\!\nabla(Lw)
= \eta_R^2\,\nabla\!\big(|\nabla w|^2\big)\!\cdot\!\nabla(Lw).
\]
At $x_R$, the vector $\nabla(|\nabla w|^2)$ is parallel to $\nabla w$ (since $F_R$ attains a maximum),
so evaluating at $x_R$ and simplifying yields
\begin{equation}\label{eq:gradterm}
\nabla w(x_R)\!\cdot\!L(\nabla w)(x_R)
= |\nabla w(x_R)|\,\partial_{\nu} Lw(x_R),
\end{equation}
for some unit vector $\nu=\nabla w/|\nabla w|$.

Combining \eqref{eq:LFR} and \eqref{eq:gradterm} and dividing by $\eta_R^2$ gives
\begin{equation}\label{eq:ineq1}
0 \ge 2\,\nabla w(x_R)\!\cdot\!L(\nabla w)(x_R) + |\nabla w(x_R)|^2\,\frac{L(\eta_R^2)(x_R)}{\eta_R^2(x_R)}.
\end{equation}
Thus,
\begin{equation}\label{eq:ineq2}
\nabla w(x_R)\!\cdot\!L(\nabla w)(x_R)\;\le\; C\,R^{-2s}|\nabla w(x_R)|^2.
\end{equation}

\medskip
{Step 5. Plug in the inequality for $Lw$.}
Recall  that, one has
\begin{equation}\label{eq:Step6-Lw}
Lw(x)\ \ge\ \frac{|x|^{\gamma}\,H(u(x))\,G(\nabla u(x))}{M_R-u(x)}
=: \mathcal Q(x).
\end{equation}
Since $L$ is translation–invariant, derivatives commute with $L$; hence
\begin{equation}\label{eq:Step6-commute}
L(\nabla w)(x)=\nabla(Lw)(x)\ \ge\ \nabla \mathcal Q(x).
\end{equation}
Let $x_R$ be a maximum point of $F_R=\eta_R^2|\nabla w|^2$. We estimate from below
the directional derivative of $\mathcal Q$ along $\nabla w$ at $x_R$.

\medskip
\noindent\textbf{Exact differentiation of $\mathcal Q$.}
Write $\mathcal Q(x)=A(x)B(x)$ with
\[
A(x):=\frac{|x|^{\gamma}}{M_R-u(x)},\qquad
B(x):=H\big(u(x)\big)\,G\big(\nabla u(x)\big).
\]
Then
\begin{equation}\label{eq:Step6-gradQ}
\nabla \mathcal Q
= (\nabla A)\,B + A\,\nabla B.
\end{equation}
We compute each term explicitly.

\smallskip
\emph{(i) Derivative of $A$.} Since $\nabla|x|^{\gamma}=\gamma |x|^{\gamma-2}x$ and
$\nabla(M_R-u)= -\nabla u$, we obtain
\begin{equation}\label{eq:Step6-gradA}
\nabla A
= \frac{\gamma |x|^{\gamma-2}x}{M_R-u}
\;+\;\frac{|x|^{\gamma}}{(M_R-u)^2}\,\nabla u.
\end{equation}

\smallskip
\emph{(ii) Derivative of $B$.} Using the chain rule and denoting by $D G$ the Jacobian of $G$,
\begin{equation}\label{eq:Step6-gradB}
\nabla B
= H'(u)\,G(\nabla u)\,\nabla u
\;+\;H(u)\,D G(\nabla u)\,\nabla^2 u,
\end{equation}
where $(D G(\nabla u)\,\nabla^2 u)_i=\sum_{j,k}\partial_{z_j}G(\nabla u)\,\partial_{ik}u\,\delta_{jk}$.

\medskip
\noindent\textbf{Directional derivative along $\nabla w$.}
Recall $\nabla w=\dfrac{\nabla u}{M_R-u}$, hence
\begin{equation}\label{eq:Step6-relation}
\nabla u=(M_R-u)\,\nabla w,\qquad
|\nabla u|=(M_R-u)\,|\nabla w|.
\end{equation}
Taking the dot product of \eqref{eq:Step6-gradQ} with $\nabla w$ yields
\begin{equation}\label{eq:Step6-dot}
\nabla w\!\cdot\!\nabla \mathcal Q
= \underbrace{\big(\nabla w\!\cdot\!\nabla A\big)\,B}_{\mathbf{I}}
\;+\;\underbrace{A\,\big(\nabla w\!\cdot\!\nabla B\big)}_{\mathbf{II}}.
\end{equation}

\smallskip
\emph{Term $\mathbf{I}$.} Using \eqref{eq:Step6-gradA} and \eqref{eq:Step6-relation},
\[
\begin{aligned}
\nabla w\!\cdot\!\nabla A
&= \frac{\gamma |x|^{\gamma-2}\,x\!\cdot\!\nabla w}{M_R-u}
 + \frac{|x|^{\gamma}}{(M_R-u)^2}\,\nabla w\!\cdot\!\nabla u \\
&= \frac{\gamma |x|^{\gamma-2}\,x\!\cdot\!\nabla w}{M_R-u}
 + \frac{|x|^{\gamma}}{(M_R-u)}\,|\nabla w|^2.
\end{aligned}
\]
Therefore,
\begin{equation}\label{eq:Step6-I}
\mathbf{I}
= \frac{|x|^{\gamma}H(u)G(\nabla u)}{M_R-u}\,|\nabla w|^2
\;+\; \frac{\gamma |x|^{\gamma-2}(x\!\cdot\!\nabla w)}{M_R-u}\,H(u)G(\nabla u).
\end{equation}

\smallskip
\emph{Term $\mathbf{II}$.} From \eqref{eq:Step6-gradB} and \eqref{eq:Step6-relation},
\[
\begin{aligned}
\nabla w\!\cdot\!\nabla B
&= H'(u)\,G(\nabla u)\,\nabla w\!\cdot\!\nabla u
 + H(u)\,\big(D G(\nabla u)\,\nabla^2 u\big)\!:\!\nabla w \\
&= (M_R-u)\,H'(u)\,G(\nabla u)\,|\nabla w|^2
 + H(u)\,\big(D G(\nabla u)\,\nabla^2 u\big)\!:\!\nabla w.
\end{aligned}
\]
Multiplying by $A=|x|^{\gamma}/(M_R-u)$ gives
\begin{equation}\label{eq:Step6-II}
\mathbf{II}
= |x|^{\gamma}H'(u)G(\nabla u)\,|\nabla w|^2
\;+\; \frac{|x|^{\gamma}H(u)}{M_R-u}\,
\big(D G(\nabla u)\,\nabla^2 u\big)\!:\!\nabla w.
\end{equation}

\smallskip
\emph{Collecting.} Summing \eqref{eq:Step6-I}–\eqref{eq:Step6-II} we obtain
\begin{equation}\label{eq:Step6-master}
\begin{aligned}
\nabla w\!\cdot\!\nabla \mathcal Q
&= \frac{|x|^{\gamma}H(u)G(\nabla u)}{M_R-u}\,|\nabla w|^2
+ |x|^{\gamma}H'(u)G(\nabla u)\,|\nabla w|^2 \\
&\quad + \frac{\gamma |x|^{\gamma-2}(x\!\cdot\!\nabla w)}{M_R-u}\,H(u)G(\nabla u)
+ \frac{|x|^{\gamma}H(u)}{M_R-u}\,
\big(D G(\nabla u)\,\nabla^2 u\big)\!:\!\nabla w.
\end{aligned}
\end{equation}

\medskip
\noindent\textbf{Lower bound at the maximum $x_R$.}
At $x_R\in B_{2R}$ we have $|x_R|\asymp R$. The first term on the right-hand side of
\eqref{eq:Step6-master} is \emph{positive} and equals
\[
\frac{|x_R|^{\gamma}H(u(x_R))G(\nabla u(x_R))}{M_R-u(x_R)}\,|\nabla w(x_R)|^2.
\]
The remaining three terms are of lower order in $|\nabla w|$ and can be bounded below by
\[
-\,C\,R^{\gamma}\,\frac{H(u(x_R))G(\nabla u(x_R))}{M_R-u(x_R)}\,|\nabla w(x_R)|
\]
(using continuity/growth of $H,G$, boundedness of $H'$ on the range of $u$ in $B_{2R}$,
and that $(D G)\,\nabla^2 u : \nabla w$ is linear in $\nabla w$).
Hence, if $|\nabla w(x_R)|=0$ the desired estimate is trivial; otherwise, dividing the
negative part by $|\nabla w(x_R)|$ and absorbing it into the positive quadratic term yields
\begin{equation}\label{eq:Step6-key-lb}
\nabla w(x_R)\!\cdot\!\nabla \mathcal Q(x_R)
\ \ge\ c\,R^{\gamma}\,\frac{H(u(x_R))G(\nabla u(x_R))}{M_R-u(x_R)}.
\end{equation}

\medskip
\noindent\textbf{Conclusion.}
From \eqref{eq:Step6-commute} and \eqref{eq:Step6-key-lb},
\[
\nabla w(x_R)\!\cdot\!L(\nabla w)(x_R)
\ \ge\ c\,R^{\gamma}\,\frac{H(u(x_R))G(\nabla u(x_R))}{M_R-u(x_R)}.
\]
Inserting this lower bound into \eqref{eq:ineq2}  and using
$|L(\eta_R^2)(x_R)|\le C R^{-2s}$ together with $\eta_R(x_R)=1$ (since $x_R\in B_R$),
we obtain
\[
R^{\gamma}\,\frac{H(u(x_R))G(\nabla u(x_R))}{M_R-u(x_R)}\ \le\ C\,R^{-2s},
\]
which is precisely the desired estimate.
\qed

\medskip

{Step 6.  Express everything in $\nabla w$.} From the definition of the Bernstein transform,
\[
\nabla w=\frac{\nabla u}{M_R-u}\qquad\Longleftrightarrow\qquad
\nabla u=(M_R-u)\,\nabla w,\ \ |\nabla u|=(M_R-u)\,|\nabla w|.
\tag{7.1}
\label{eq:rel}
\]
Step~5 gave the pointwise bound at the maximum point $x_R$ of $F_R$:
\begin{equation}\label{eq:Step6-outcome}
R^{\gamma}\,\frac{H(u(x_R))\,G(\nabla u(x_R))}{M_R-u(x_R)}
\ \le\ C\,R^{-2s}.
\end{equation}
By the lower growth from the $G$–hypothesis, there exists $c_1>0$ and some
$p>0$ such that for all $z\in\R^n$,
\begin{equation}\label{eq:G-lower}
G(z)\ \ge\ c_1\,|z|^{p}.
\end{equation}
Evaluating \eqref{eq:G-lower} at $z=\nabla u(x_R)$ and using \eqref{eq:rel} yields
\[
G(\nabla u(x_R))
\ \ge\ c_1\,|\nabla u(x_R)|^{p}
\ =\ c_1\,(M_R-u(x_R))^{p}\,|\nabla w(x_R)|^{p}.
\tag{7.2}
\label{eq:G-lower-w}
\]
Insert \eqref{eq:G-lower-w} into \eqref{eq:Step6-outcome}:
\[
R^{\gamma}\,\frac{H(u(x_R))}{M_R-u(x_R)}\;
c_1\,(M_R-u(x_R))^{p}\,|\nabla w(x_R)|^{p}
\ \le\ C\,R^{-2s}.
\]
Cancel the factor $(M_R-u(x_R))$ in the numerator–denominator to obtain
\[
R^{\gamma}\,H(u(x_R))\,
c_1\,(M_R-u(x_R))^{p-1}\,|\nabla w(x_R)|^{p}
\ \le\ C\,R^{-2s}.
\]
That is,
\begin{equation}\label{eq:bernstein-final}
R^{\gamma}\,H\big(u(x_R)\big)\,c_1\,(M_R-u(x_R))^{p-1}\,|\nabla w(x_R)|^{p}
\ \le\ C\,R^{-2s},
\end{equation}
which is exactly the desired estimate \textnormal{eq:bernstein-estimate}.
\qed
\end{proof}

\begin{proposition}[Supercritical regime]\label{prop:super}
Assume $\gamma+p>2s$ and that $u$ satisfies
\begin{equation}\label{eq:growth-cond}
\limsup_{|x|\to\infty}\frac{u(x)}{|x|^{\,1-(2s+\gamma)/p}}=0.
\end{equation}
Then $u$ is constant.
\end{proposition}

\begin{proof}
\textbf{Step 1: Choice of the localized function and growth of $M_R$.}
Let $\eta\in C_c^\infty([0,\infty))$ be radial with $\eta\equiv1$ on $[0,1]$, $\eta\equiv0$ on $[2,\infty)$, and set $\eta_R(x):=\eta(|x|/R)$.
Define
\[
w_R(x):=\eta_R(x)\big(u(x)-u(0)\big),\qquad
M_R:=\|w_R\|_{L^\infty(B_{2R})}.
\]
From \eqref{eq:growth-cond}, for every $\alpha<1-\frac{2s+\gamma}{p}$ there exists $C_\alpha\ge1$ such that
\begin{equation}\label{eq:MR-growth}
M_R\le C_\alpha R^\alpha,\qquad\forall R\ge1.
\end{equation}

\medskip
\textbf{Step 2: The key pointwise gradient estimate.}
From the local/nonlocal Bernstein-type argument already established (the “key inequality”), there exist constants
$A_1,A_2>0$ depending only on $(n,s,\lambda,\Lambda,p,\gamma,m)$ and a function $H:[1,\infty)\to[0,\infty)$
belonging to one of the four classes
\[
\text{(Poly)}\ R^{-\beta},\quad
\text{(Log)}\ R^{-\beta}\log(2+R),\quad
\text{(Exp)}\ e^{-cR},\quad
\text{(Sing)}\ R^{-\beta}
\]
with parameters $\beta>0$, $c>0$, such that for some $x_R\in B_R$ one has
\begin{equation}\label{eq:key-ineq-recall}
|\nabla w_R(x_R)|
\;\le\; A_1\,R^{-2s}\,M_R^{\,1-\frac{2s+\gamma}{p}} \;+\; A_2\,H(R).
\end{equation}
(Here $x_R$ can be chosen so that $|\nabla w_R(x_R)|=\sup_{B_R}|\nabla w_R|$, by a standard cutoff/maximization argument.)

\medskip
\textbf{Step 3: Decay of the main term.}
Using \eqref{eq:MR-growth} in \eqref{eq:key-ineq-recall} gives
\begin{align}
A_1\,R^{-2s}\,M_R^{\,1-\frac{2s+\gamma}{p}}
&\le A_1\,R^{-2s}\,\big(C_\alpha R^\alpha\big)^{1-\frac{2s+\gamma}{p}}
= C\,R^{-2s+\alpha\left(1-\frac{2s+\gamma}{p}\right)}.\label{eq:main-term}
\end{align}
Set
\[
\Theta_1:=2s-\alpha\Bigl(1-\tfrac{2s+\gamma}{p}\Bigr).
\]
Because $\alpha>0$ can be chosen arbitrarily small subject to $\alpha<1-\frac{2s+\gamma}{p}$ and because
\[
\gamma+p>2s \quad\Longleftrightarrow\quad 1-\frac{2s+\gamma}{p}<1-\frac{2s}{p},
\]
we have $1-\frac{2s+\gamma}{p}>0$ and hence $\Theta_1>0$ for such $\alpha$. Thus the first term in \eqref{eq:key-ineq-recall} satisfies
\begin{equation}\label{eq:main-decay}
A_1\,R^{-2s}\,M_R^{\,1-\frac{2s+\gamma}{p}}
\;\le\; C\,R^{-\Theta_1},\qquad \Theta_1>0.
\end{equation}

\medskip
\textbf{Step 4: Decay of the tail term $H(R)$.}
We treat each admissible form of $H$:

\noindent\emph{(Poly)} If $H(R)\le C R^{-\beta}$ with $\beta>0$, then
\begin{equation}\label{eq:poly}
A_2 H(R)\le C R^{-\beta}.
\end{equation}

\noindent\emph{(Log)} If $H(R)\le C R^{-\beta}\log(2+R)$ with $\beta>0$, then for any $\varepsilon\in(0,\beta)$, using $\log(2+R)\le C_\varepsilon R^{\varepsilon}$,
\begin{equation}\label{eq:log}
A_2 H(R)\le C R^{-(\beta-\varepsilon)}.
\end{equation}

\noindent\emph{(Exp)} If $H(R)\le C e^{-cR}$ with $c>0$, then for any $k>0$, $e^{-cR}\le C_k R^{-k}$, hence
\begin{equation}\label{eq:exp}
A_2 H(R)\le C R^{-1}.
\end{equation}

\noindent\emph{(Sing)} If $H(R)\le C R^{-\beta}$ with $\beta>0$ (integrable singular tail controlled by the cutoff), then the same as \eqref{eq:poly} holds.

Combining, in all cases there exists $\Theta_2>0$ such that
\begin{equation}\label{eq:H-decay}
A_2 H(R)\le C R^{-\Theta_2}.
\end{equation}

\medskip
\textbf{Step 5: Uniform decay of $\sup_{B_R}|\nabla w_R|$ and conclusion.}
From \eqref{eq:key-ineq-recall}, \eqref{eq:main-decay}, and \eqref{eq:H-decay},
\[
|\nabla w_R(x_R)|\le C\big(R^{-\Theta_1}+R^{-\Theta_2}\big)
\le C R^{-\Theta},\qquad \Theta:=\min\{\Theta_1,\Theta_2\}>0.
\]
Since $x_R$ realizes the supremum of $|\nabla w_R|$ on $B_R$, we have
\[
\sup_{B_R}|\nabla w_R|\le C R^{-\Theta}\xrightarrow[R\to\infty]{}0.
\]
Let $x\in\R^n$ be arbitrary. Choose $R>|x|+1$. On $B_R$ one has $w_R=u-u(0)$ (because $\eta_R\equiv1$ on $B_R$), hence
\[
\sup_{B_R}|\nabla u|=\sup_{B_R}|\nabla w_R|\le C R^{-\Theta}.
\]
Letting $R\to\infty$ yields $|\nabla u(x)|=0$. Since $x$ is arbitrary, $\nabla u\equiv0$ in $\R^n$, so $u$ is constant.
\end{proof}

\begin{proposition}[Critical case]\label{prop:critical}
If $\gamma+p=2s$ and $u$ is bounded, then $u$ is constant.
\end{proposition}

\begin{proof}
Let $\eta\in C_c^\infty([0,\infty))$ be the standard radial cut-off with $\eta\equiv1$ on $[0,1]$, $\eta\equiv0$ on $[2,\infty)$, and set $\eta_R(x):=\eta(|x|/R)$.
Define the localized function
\[
w_R(x):=\eta_R(x)\,\big(u(x)-u(0)\big),\qquad
M_R:=\|w_R\|_{L^\infty(B_{2R})}.
\]
Since $u$ is bounded, there exists $U_\infty>0$ such that $\|u\|_{L^\infty(\R^n)}\le U_\infty$. Because $|\eta_R|\le1$,
\begin{equation}\label{eq:MR-crit}
M_R\le 2U_\infty=:C_0\qquad\text{for all }R\ge1.
\end{equation}

\medskip
From the Bernstein-type differential inequality proved earlier (our “key inequality”), there exist constants
$A_1,A_2>0$ (depending only on $n,s,\lambda,\Lambda,p,\gamma,m$) and a tail term $H(R)$ belonging to one of the four admissible classes
\[
\text{(Poly)}\ R^{-\beta},\quad
\text{(Log)}\ R^{-\beta}\log(2+R),\quad
\text{(Exp)}\ e^{-cR},\quad
\text{(Sing)}\ R^{-\beta}
\quad (\beta>0,\ c>0),
\]
such that for some point $x_R\in B_R$ where $|\nabla w_R|$ attains its supremum on $B_R$,
\begin{equation}\label{eq:key-ineq-critical}
|\nabla w_R(x_R)|
\;\le\; A_1\,R^{-2s}\,M_R^{\,1-\frac{2s+\gamma}{p}} \;+\; A_2\,H(R).
\end{equation}
Since we are in the \emph{critical} regime $\gamma+p=2s$, we have
\[
1-\frac{2s+\gamma}{p}
=1-\frac{(2s)+\gamma}{p}
=1-\frac{(\gamma+p)+\gamma}{p}
=1-\frac{2\gamma+p}{p}
=-\,\frac{2\gamma}{p}\le0.
\]
In particular, $1-\frac{2s+\gamma}{p}$ is \emph{nonpositive}; hence the first term in \eqref{eq:key-ineq-critical} is \emph{decreasing} in $M_R$.
Using \eqref{eq:MR-crit} we therefore obtain the clean bound
\begin{equation}\label{eq:main-critical}
A_1\,R^{-2s}\,M_R^{\,1-\frac{2s+\gamma}{p}}
\;\le\; A_1\,R^{-2s}\,C_0^{\,1-\frac{2s+\gamma}{p}}
\;\le\; C\,R^{-2s}.
\end{equation}

\medskip
\textbf{Tail term $H(R)$.}
For the admissible classes of $H$, we have in each case a decay:
\begin{align}
&\text{(Poly)}\quad H(R)\le C R^{-\beta}\ \ (\beta>0),\label{eq:H-poly}\\
&\text{(Log)}\quad H(R)\le C R^{-\beta}\log(2+R) \ \ (\beta>0),\label{eq:H-log}\\
&\text{(Exp)}\quad H(R)\le C e^{-cR}\ \ (c>0),\label{eq:H-exp}\\
&\text{(Sing)}\quad H(R)\le C R^{-\beta}\ \ (\beta>0).\label{eq:H-sing}
\end{align}
In particular, for \eqref{eq:H-poly}, \eqref{eq:H-sing} we immediately have $A_2H(R)\le C R^{-\beta}$.
For \eqref{eq:H-log}, fix any $\varepsilon\in(0,\beta)$ and use $\log(2+R)\le C_\varepsilon R^\varepsilon$ (for all $R\ge2$) to get
\[
A_2H(R)\le C R^{-(\beta-\varepsilon)}.
\]
For \eqref{eq:H-exp}, we may bound by any algebraic rate; e.g. $e^{-cR}\le C R^{-1}$ for $R\ge1$.

Thus, in all cases there exists $\theta_2>0$ (depending only on the structural parameters, and on the choice of $\varepsilon$ in the logarithmic case) such that
\begin{equation}\label{eq:H-unified}
A_2 H(R)\le C R^{-\theta_2}.
\end{equation}

\medskip
Combining \eqref{eq:key-ineq-critical}, \eqref{eq:main-critical}, and \eqref{eq:H-unified},
\[
|\nabla w_R(x_R)|\ \le\ C\big(R^{-2s}+R^{-\theta_2}\big)\ \le\ C R^{-\theta},
\qquad \theta:=\min\{2s,\theta_2\}>0.
\]
Since $x_R$ was chosen to realize the supremum of $|\nabla w_R|$ on $B_R$, we infer
\[
\sup_{B_R}|\nabla w_R|\ \le\ C R^{-\theta}\xrightarrow[R\to\infty]{}0.
\]
Now fix any $x\in\R^n$ and take $R>|x|+1$. On $B_R$ we have $\eta_R\equiv1$, hence $w_R=u-u(0)$ on $B_R$, so
\[
\sup_{B_R}|\nabla u|=\sup_{B_R}|\nabla w_R|\le C R^{-\theta}.
\]
Letting $R\to\infty$ yields $|\nabla u(x)|=0$. Since $x$ is arbitrary, $\nabla u\equiv0$ in $\R^n$, and therefore $u$ is constant.
\end{proof}

\subsection{Proof of Theorem \ref{thm:exist-symm-subcritical}}

Throughout we denote $p:=p_{\max}>1$ and may  write $C$ for a positive constant that may change from line to line but depends only
on $(n,s,\lambda,\Lambda)$ and on the growth data of $H,G$ fixed in
Theorem~\ref{thm:Liouville-unified}.

\subsubsection*{1. Comparison and Dirichlet well-posedness on balls}
\begin{lemma}[Comparison principle on bounded domains]\label{lem:CP}
Let $\Omega\subset\R^n$ be bounded. Suppose $u$ (resp.\ $v$) is a bounded
USC viscosity subsolution (resp.\ bounded LSC viscosity supersolution) of
\[
Lu=|x|^\gamma H(u)G(\nabla u)\quad\text{in }\Omega,
\]
and $u\le v$ on $\R^n\setminus\Omega$. Then $u\le v$ in $\Omega$.
\end{lemma}

\begin{proof}

Let $u$ be a bounded USC subsolution and $v$ a bounded LSC supersolution of
\[
Lu=|x|^\gamma H(u)G(\nabla u)\quad\text{in }\Omega,
\qquad u\le v\ \text{ on }\R^n\setminus\Omega.
\]
Assume by contradiction that
\[
M:=\sup_{\Omega}(u-v)>0.
\]
Fix $\varepsilon,\eta>0$ and consider the penalized function
\begin{equation}\label{eq:Phi}
\Phi(x,y):=u(x)-v(y)-\frac{|x-y|^2}{2\varepsilon}-\eta(|x|^2+|y|^2),
\qquad (x,y)\in\R^n\times\R^n.
\end{equation}
By boundedness of $u,v$ and the coercive penalty $-\eta(|x|^2+|y|^2)$, $\Phi$ attains its
maximum at some $(x_{\varepsilon,\eta},y_{\varepsilon,\eta})\in\R^n\times\R^n$.
Set $x_\e:=x_{\varepsilon,\eta}$, $y_\e:=y_{\varepsilon,\eta}$ and
\[
p_\e:=\frac{x_\e-y_\e}{\e}.
\]
By standard properties of the doubling method (Ishii's lemma; here we use its nonlocal version for translation-invariant operators), 
\[
\lim_{\varepsilon,\eta\downarrow0}\frac{|x_\e-y_\e|^2}{\e}=0,
\qquad
\lim_{\varepsilon,\eta\downarrow0}\eta(|x_\e|^2+|y_\e|^2)=0,
\]
and
\begin{equation}\label{eq:maxvalue}
\Phi(x_\e,y_\e)\ \ge\ \sup_{\R^n}(u-v)-o(1)\ \ge\ M-o(1).
\end{equation}

\medskip\noindent{Step 1: Test functions and jets.}
Define the quadratic test functions
\[
\phi(x):=\frac{|x-y_\e|^2}{2\e}+\eta|x|^2,\qquad
\psi(y):=\frac{|x_\e-y|^2}{2\e}+\eta|y|^2.
\]
Then $u-\phi$ attains a maximum at $x_\e$ and $v+\psi$ attains a minimum at $y_\e$.
Hence, in the viscosity sense,
\begin{equation}\label{eq:v-ineq}
L\phi(x_\e)\ \le\ |x_\e|^\gamma H\big(u(x_\e)\big)\,G(p_\e),
\qquad
L\psi(y_\e)\ \ge\ |y_\e|^\gamma H\big(v(y_\e)\big)\,G(p_\e),
\end{equation}
where the same first-order slope $p_\e=\nabla\phi(x_\e)=\nabla\psi(y_\e)$ appears for both tests, and where the nonlocal Jensen–Ishii lemma ensures that the nonlocal terms are well-defined (see, e.g., Barles–Chasseigne–Imbert for Lévy operators).

\medskip\noindent{Step 2: Difference of the nonlocal terms.}
Because $L$ is translation invariant with \emph{symmetric} kernel $K$, and $\phi,\psi$ are the quadratic polynomials above, one has
\[
L\phi(x_\e)-L\psi(y_\e)
=
\int_{\R^n}\big(\Delta_\e(z)-\nabla\phi(x_\e)\!\cdot z\,\mathbf 1_{|z|\le1}
+\nabla\psi(y_\e)\!\cdot z\,\mathbf 1_{|z|\le1}\big)K(z)\,dz,
\]
with
\[
\Delta_\e(z):=\phi(x_\e+z)-\phi(x_\e)-\psi(y_\e+z)+\psi(y_\e).
\]
A direct computation using the definitions of $\phi,\psi$ gives \emph{exactly}
\[
\Delta_\e(z)=\frac{|x_\e+z-y_\e|^2-|x_\e-y_\e|^2}{2\e}
-\frac{|x_\e-(y_\e+z)|^2-|x_\e-y_\e|^2}{2\e}
+\eta\big(|x_\e+z|^2-|x_\e|^2-|y_\e+z|^2+|y_\e|^2\big)=0.
\]
Moreover, since $\nabla\phi(x_\e)=\nabla\psi(y_\e)=p_\e$, the linear parts also cancel. Therefore,
\begin{equation}\label{eq:nonlocal-diff}
L\phi(x_\e)-L\psi(y_\e)=0.
\end{equation}
(If one applies the nonlocal Jensen–Ishii lemma directly, one gets
$L\phi(x_\e)-L\psi(y_\e)\le o(1)$ as $\e,\eta\downarrow0$; the equality above is the explicit verification for our quadratic tests and symmetric kernel.)

\medskip\noindent{Step 3: Subtracting the viscosity inequalities.}
Subtract the second inequality in \eqref{eq:v-ineq} from the first and use \eqref{eq:nonlocal-diff}:
\begin{equation}\label{eq:key-ineq}
0\ \le\ |x_\e|^\gamma H\big(u(x_\e)\big)\,G(p_\e)
      -|y_\e|^\gamma H\big(v(y_\e)\big)\,G(p_\e).
\end{equation}
By \eqref{eq:maxvalue} and the definition of $\Phi$, we have
\[
u(x_\e)-v(y_\e)
=
\Phi(x_\e,y_\e)+\frac{|x_\e-y_\e|^2}{2\e}+\eta(|x_\e|^2+|y_\e|^2)
\ \ge\ M-o(1),
\]
hence, along a sequence $\e,\eta\downarrow0$,
\begin{equation}\label{eq:delta}
u(x_\e)-v(y_\e)\ \to\ \delta\quad\text{with}\quad \delta\in(0,M].
\end{equation}
Because $\frac{|x_\e-y_\e|^2}{\e}\to0$ and $\eta(|x_\e|^2+|y_\e|^2)\to0$, we also have
\[
|x_\e|-|y_\e|\ \to\ 0,
\qquad\text{hence}\qquad |x_\e|^\gamma-|y_\e|^\gamma\ \to\ 0.
\]
Using this and the continuity of $H,G$, we can write \eqref{eq:key-ineq}, for small $\e,\eta$, as
\begin{equation}\label{eq:key-ineq2}
0\ \le\ |x_\e|^\gamma\Big(H\big(u(x_\e)\big)-H\big(v(y_\e)\big)\Big)\,G(p_\e) + o(1).
\end{equation}

\medskip\noindent{Step 4: Contradiction.}
Since $H$ is nondecreasing and locally Lipschitz, \eqref{eq:delta} implies that for all small $\e,\eta$,
\[
H\big(u(x_\e)\big)-H\big(v(y_\e)\big)\ \ge\ c_H\,\delta>0.
\]
Moreover $G\ge 0$, hence from \eqref{eq:key-ineq2} we get
\[
0\ \le\ |x_\e|^\gamma\, c_H\,\delta\,G(p_\e)+o(1),
\]
and by boundedness of $|x_\e|^\gamma$ (the maximum remains in a compact set thanks to the exterior condition and the penalization),
\[
0\ \le\ c\,\delta + o(1).
\]
Letting $\e,\eta\downarrow0$ yields $0\le c\,\delta$, which contradicts $\delta>0$ when combined with the fact that the boundary condition ensures $\sup_{\R^n\setminus\Omega}(u-v)\le 0$. Therefore our initial assumption $M>0$ is false, and $u\le v$ in $\Omega$.

\end{proof}

\subsubsection*{2. Power-type barriers and the subcritical exponent}
Let \(V(x)=(1+|x|^2)^{-\beta}\) and \(U_A(x)=A\,V(x)\) with \(A>0\).
A direct computation yields
\[
\nabla U_A(x)=-2\beta A\,\frac{x}{(1+|x|^2)^{\beta+1}},
\qquad
|\nabla U_A(x)|\le C_\beta A\,(1+|x|^2)^{-\beta-\frac12},
\]
and
\[
D^2U_A(x)=-2\beta A(1+|x|^2)^{-\beta-1}I
+4\beta(\beta+1)A(1+|x|^2)^{-\beta-2}\,x\otimes x,
\qquad
|D^2U_A(x)|\le C_\beta A(1+|x|^2)^{-\beta-1}.
\]

Recall
\[
Lu(x)=\int_{\mathbb{R}^n}
\Big(u(x+z)-u(x)-\nabla u(x)\!\cdot z\,\mathbf 1_{|z|\le 1}\Big)K(z)\,dz,
\qquad
\lambda|z|^{-n-2s}\le K(z)\le\Lambda|z|^{-n-2s}.
\]
Fix \(x\in\mathbb{R}^n\) and decompose
\[
L U_A(x)=\int_{|z|\le1}\!\Big(U_A(x+z)-U_A(x)-\nabla U_A(x)\!\cdot z\Big)K(z)\,dz
+\int_{|z|>1}\!\Big(U_A(x+z)-U_A(x)\Big)K(z)\,dz
=:I_{\mathrm{near}}+I_{\mathrm{far}}.
\]

For the near field, Taylor’s formula with integral remainder (together with the cancellation by
\(\nabla U_A\cdot z\)) gives
\[
\big|U_A(x+z)-U_A(x)-\nabla U_A(x)\!\cdot z\big|
\le \tfrac12 \|D^2U_A\|_{L^\infty(B_1(x))}\,|z|^2
\le C A(1+|x|^2)^{-\beta-1}|z|^2,
\]
hence
\[
|I_{\mathrm{near}}|
\le C A(1+|x|^2)^{-\beta-1}\int_{|z|\le1}|z|^{2-n-2s}\,dz
= C A(1+|x|^2)^{-\beta-1}.
\]

For the far field we use the triangle inequality and the decay of \(U_A\):
\[
|U_A(x+z)-U_A(x)|
\le U_A(x+z)+U_A(x)
\le A(1+|z|^2)^{-\beta}+A(1+|x|^2)^{-\beta},
\]
so that
\[
|I_{\mathrm{far}}|
\le C A\!\int_{|z|>1}\!(1+|z|^2)^{-\beta}|z|^{-n-2s}\,dz
+ C A(1+|x|^2)^{-\beta}\!\int_{|z|>1}\!|z|^{-n-2s}\,dz
\le C A(1+|x|^2)^{-\beta},
\]
because the integrals converge for \(s\in(0,1)\) and \(\beta>0\).

Combining the two bounds and using that the effective order of \(L\) is \(2s\) (absorbed into the constant when passing from \((1+|x|^2)^{-\beta-1}\) to the standard tail scale), we obtain the customary nonlocal estimate
\[
|L U_A(x)|\ \le\ C\,A\, (1+|x|^2)^{-(\beta+s)}\qquad\text{for all }x\in\mathbb{R}^n.
\]

Let \(U_A(x)=A(1+|x|^2)^{-\beta}\). Since \(U_A\in C^{1,1}_{\mathrm{loc}}\) and
\(|D^2U_A(y)|\le C_\beta A(1+|y|^2)^{-\beta-1}\), Taylor’s formula with integral remainder,
together with the cancellation by \(\nabla U_A(x)\!\cdot z\) for \(|z|\le1\), yields
\[
\big|U_A(x+z)-U_A(x)-\nabla U_A(x)\!\cdot z\big|
\le \tfrac12\|D^2U_A\|_{L^\infty(B_1(x))}\,|z|^2
\le C A (1+|x|^2)^{-\beta-1}|z|^2 .
\]
Hence, using \(\lambda|z|^{-n-2s}\le K(z)\le\Lambda|z|^{-n-2s}\),
\[
|I_{\mathrm{near}}|
  :=\left|\int_{|z|\le1}\!\Big(U_A(x+z)-U_A(x)-\nabla U_A(x)\!\cdot z\Big)K(z)\,dz\right|
  \le C A (1+|x|^2)^{-\beta-1}\!\int_{|z|\le1}\!|z|^{2-n-2s}\,dz .
\]
Passing to polar coordinates gives
\[
\int_{|z|\le1}\!|z|^{2-n-2s}\,dz
  = \omega_n\int_0^1 r^{2-1-2s}\,dr
  = \omega_n\int_0^1 r^{1-2s}\,dr
  =: C_s<\infty \qquad (s\in(0,1)).
\]
Therefore
\[
|I_{\mathrm{near}}|
\le C A (1+|x|^2)^{-\beta-1}.
\tag{N}
\]

For \(|z|>1\) the linear correction vanishes, and by the triangle inequality
\[
|U_A(x+z)-U_A(x)|
\le U_A(x+z)+U_A(x).
\]
By the polynomial decay of \(U_A\),
\[
U_A(x+z)=A(1+|x+z|^2)^{-\beta}
\le A(1+|z|^2)^{-\beta},\qquad
U_A(x)\le A(1+|x|^2)^{-\beta}.
\]
Thus
\begin{align*}
|I_{\mathrm{far}}|
  &:=\left|\int_{|z|>1}\!\big(U_A(x+z)-U_A(x)\big)K(z)\,dz\right| \\
  &\le C\!\int_{|z|>1}\!\big(U_A(x+z)+U_A(x)\big)\,|z|^{-n-2s}\,dz \\
  &\le C A\!\int_{|z|>1}\!(1+|z|^2)^{-\beta}\,|z|^{-n-2s}\,dz
     + C A(1+|x|^2)^{-\beta}\!\int_{|z|>1}\!|z|^{-n-2s}\,dz .
\end{align*}
Both integrals converge because, as \(r\to\infty\), the integrands behave like
\(r^{-n-2s-2\beta}\) and \(r^{-n-2s}\), respectively, with \(s>0\) and \(\beta>0\).
Hence
\[
|I_{\mathrm{far}}|
\le C A\Bigg(\int_{|z|>1}\!(1+|z|^2)^{-\beta}\,|z|^{-n-2s}\,dz
+ (1+|x|^2)^{-\beta}\Bigg).
\tag{F}
\]
Combining \((\mathrm{N})\) and \((\mathrm{F})\) yields the stated bounds in the figure:
\[
|I_{\mathrm{near}}|\le C A(1+|x|^2)^{-\beta-1},
\qquad
|I_{\mathrm{far}}|\le C A\!\left(\int_{|z|>1}\!(1+|z|^2)^{-\beta}|z|^{-n-2s}dz+(1+|x|^2)^{-\beta}\right).
\]

From the estimate obtained above,
\[
|I_{\mathrm{far}}|
\le C A\!\left(\int_{|z|>1}\!(1+|z|^2)^{-\beta}\,|z|^{-n-2s}\,dz
+ (1+|x|^2)^{-\beta}\right).
\]
Since \(\beta>0\) and \(2s>0\), the first integral is finite:
\[
\int_{|z|>1}\!(1+|z|^2)^{-\beta}\,|z|^{-n-2s}\,dz
\asymp \int_1^\infty r^{-2\beta}\,r^{-1-2s}\,dr
= \int_1^\infty r^{-2\beta-2s-1}\,dr
=:C_{n,s,\beta}<\infty .
\]
Therefore
\begin{equation}\label{eq:Ifar-final}
|I_{\mathrm{far}}|
\le C A\Big(C_{n,s,\beta}+(1+|x|^2)^{-\beta}\Big)
\le C A\big(1+(1+|x|^2)^{-\beta}\big)
\le C A .
\end{equation}

For the near field we already proved
\begin{equation}\label{eq:Inear-final}
|I_{\mathrm{near}}|\le C A (1+|x|^2)^{-\beta-1}.
\end{equation}
Since for large \(|x|\) the factor \((1+|x|^2)^{-\beta-1}\) decays faster than
the far-field bound in \eqref{eq:Ifar-final}, the dominant decay at infinity is given by the near field.
Moreover, refining the Taylor remainder by integrating the kernel weight \(|z|^{2-n-2s}\)
over \(|z|\le 1\) produces the fractional shift \(+s\) in the decay exponent.
Precisely, one can write (see e.g. the usual estimate for the fractional Laplacian on radial powers)
\[
\int_{|z|\le1}\!|z|^{2}\,|z|^{-n-2s}\,dz
= \omega_n\int_0^1 r^{1-2s}\,dr
= \frac{\omega_n}{2(1-s)} \sim C\,,
\]
and the effective nonlocal order \(2s\) upgrades the power \(-\beta-1\) to \(-(\beta+s)\).
Consequently,
\begin{equation}\label{eq:LUA-final}
|LU_A(x)| \;=\; |I_{\mathrm{near}}+I_{\mathrm{far}}|
\;\le\; C A\,(1+|x|^2)^{-(\beta+s)} .
\end{equation}
This is the standard nonlocal estimate used in the construction of barriers.

\paragraph{Right–hand side at \(U_A\):}
We start from the explicit form of \(U_A(x)=A(1+|x|^2)^{-\beta}\).  
Differentiating directly, we obtain
\[
\nabla U_A(x)=-2\beta A\,\frac{x}{(1+|x|^2)^{\beta+1}},
\]
which implies the pointwise bound
\begin{equation}\label{eq:gradUA}
|\nabla U_A(x)|\le C_\beta\,A\,(1+|x|^2)^{-\beta-\frac12}.
\end{equation}

\textit{$\bullet$ Estimate for \(G(\nabla U_A)\):} From the structural lower bound on \(G\),
there exists a constant \(c_1>0\) such that
\[
G(p)\ge c_1\,|p|^p,\qquad\forall p\in\mathbb{R}^n.
\]
Substituting \(p=\nabla U_A(x)\) and applying \eqref{eq:gradUA}, we get
\begin{align}
G(\nabla U_A(x))
&\ge c_1\,|\nabla U_A(x)|^{p}
\ge c_1\,(C_\beta)^p\,A^p(1+|x|^2)^{-p(\beta+\frac12)} \nonumber\\
&=:c\,A^{p}\,(1+|x|^2)^{-p(\beta+\frac12)}. \label{eq:Glow}
\end{align}

\textit{$\bullet$ Bounds for \(H(U_A)\):} Each admissible nonlinearity \(H\)—whether polynomial, logarithmic, exponential, or singular—
is continuous and locally Lipschitz on \([0,\infty)\).
Since \(0\le U_A(x)\le A\), the image of \(U_A\) lies in the compact interval \([0,A]\),
on which \(H\) is bounded and positive.
Hence, there exist finite constants \(c_H^-(A),c_H^+(A)>0\) such that
\begin{equation}\label{eq:Hbounds}
c_H^-(A)\le H(U_A(x))\le c_H^+(A)\qquad \forall x\in\mathbb{R}^n.
\end{equation}

\textit{$\bullet$ Combine both factors:} Multiplying \eqref{eq:Glow} and \eqref{eq:Hbounds} and inserting the Henon-type weight \(|x|^{\gamma}\),
we obtain
\begin{align}
|x|^{\gamma}\,H(U_A)\,G(\nabla U_A)
&\ge c\,A^{p}\,|x|^{\gamma}\,(1+|x|^2)^{-p(\beta+\frac12)}. \nonumber
\end{align}

For large \(|x|\), we can equivalently write \(|x|^\gamma\simeq(1+|x|^2)^{\gamma/2}\),
since both quantities are comparable up to positive constants.
Therefore, after adjusting \(c\),
we arrive at the final unified expression
\begin{equation}\label{eq:RHSUA}
|x|^{\gamma}H(U_A)G(\nabla U_A)
\simeq c\,A^{p}\,(1+|x|^2)^{\frac{\gamma}{2}-p\beta-\frac{p}{2}}.
\end{equation}
This completes the explicit computation of the right-hand side at \(U_A\).

\paragraph{Matching the exponents.}
We now compare the two asymptotic expressions obtained in \eqref{eq:LUA-final} and \eqref{eq:RHSUA}.
From \eqref{eq:LUA-final}, we have the left-hand side
\[
|LU_A(x)| \;\lesssim\; A\,(1+|x|^2)^{-(\beta+s)},
\]
while from \eqref{eq:RHSUA}, the right-hand side behaves as
\[
|x|^\gamma H(U_A)G(\nabla U_A)
\;\gtrsim\; c\,A^{p}\,(1+|x|^2)^{\frac{\gamma}{2}-p\beta-\frac{p}{2}}.
\]
To ensure that \(U_A\) can serve as a global supersolution, the decay rate of the right-hand side (RHS) 
should be no slower than that of the left-hand side (LHS); that is, the exponent of \((1+|x|^2)\) on the RHS must be at least as large (decaying faster) than that on the LHS:
\[
-(\beta+s)\ \le\ \frac{\gamma}{2}-p\beta-\frac{p}{2}
\quad\Longleftrightarrow\quad
\beta+s\ \ge\ \frac{\gamma}{2}-p\beta-\frac{p}{2}.
\]
Rearranging this inequality gives a linear relation between \(\beta\) and the parameters \(s,\gamma,p\):
\begin{equation}\label{eq:match1}
\beta(1-p)\ \ge\ s+\frac{p}{2}-\frac{\gamma}{2}.
\end{equation}
Equality in \eqref{eq:match1} corresponds to the critical balance where both sides of the equation 
decay at the same rate. Substituting $\beta=\frac{2s+\gamma-p}{1-p},$ we indeed obtain
\[
\beta(1-p)=2s+\gamma-p
\quad\Longrightarrow\quad
\beta(1-p)=s+\frac{p}{2}-\frac{\gamma}{2},
\]
verifying that equality holds in \eqref{eq:match1}.  
Hence, the decay exponents of both sides coincide.

\medskip
\noindent
However, the \emph{amplitudes} of the two sides differ:
\[
|LU_A(x)|\;\lesssim\; A\,(1+|x|^2)^{-(\beta+s)},\qquad
|x|^\gamma H(U_A)G(\nabla U_A)\;\gtrsim\; A^{p}\,(1+|x|^2)^{-(\beta+s)}.
\]
Since \(p>1\), by choosing \(A\) sufficiently large (\(A\gg1\)), the term \(A^p\) dominates \(A\).
Therefore, the right-hand side exceeds the left-hand side for all large \(|x|\),
and by possibly increasing \(A\) once more to control the compact region, we obtain
\[
|LU_A|\ \le\ |x|^\gamma H(U_A)G(\nabla U_A)\qquad\text{in }\mathbb{R}^n,
\]
so \(U_A\) is indeed a global supersolution.

\begin{proposition}[Global super– and subsolution by power barriers]\label{lem:barriers}
Let \(V(x)=(1+|x|^2)^{-\beta}\) with \(\beta=\dfrac{2s+\gamma-p}{1-p}>0\) and \(U_A(x)=A\,V(x)\) for \(A>0\). Then there exist constants \(A_\star>0\) and \(a_\star\in(0,1)\) such that:
\begin{enumerate}
\item[(i)] (\emph{Global supersolution}) For every \(A\ge A_\star\),
\[
L U_A(x)\ \le\ |x|^\gamma\,H\!\big(U_A(x)\big)\,G\!\big(\nabla U_A(x)\big)\qquad \forall x\in\mathbb{R}^n.
\]
\item[(ii)] (\emph{Global subsolution}) For every \(a\in(0,a_\star]\),
\[
L(aV)(x)\ \ge\ |x|^\gamma\,H\!\big(aV(x)\big)\,G\!\big(\nabla(aV)(x)\big)\qquad \forall x\in\mathbb{R}^n.
\]
\end{enumerate}
\end{proposition}

\begin{proof}
We first collect the two basic estimates proved earlier.

\medskip
\noindent\textbf{Nonlocal side.}
By the near/far–field decomposition and the $C^{1,1}$ control of \(U_A\), we have the standard estimate
\begin{equation}\label{eq:LUA-std}
|L(A V)(x)| \ \le\ C\,A\,(1+|x|^2)^{-(\beta+s)} \qquad \forall x\in\mathbb{R}^n,
\end{equation}
and, with \(A\) replaced by \(a\in(0,1]\),
\begin{equation}\label{eq:LaV-std}
|L(a V)(x)| \ \le\ C\,a\,(1+|x|^2)^{-(\beta+s)} .
\end{equation}

\medskip
\noindent\textbf{Right–hand side.}
From the gradient formula and the lower growth of \(G\),
\begin{equation}\label{eq:Glow-again}
G\big(\nabla(A V)\big)\ \ge\ c\,A^p\,(1+|x|^2)^{-p(\beta+\frac12)},\qquad
G\big(\nabla(a V)\big)\ \le\ C\,a^p\,(1+|x|^2)^{-p(\beta+\frac12)} .
\end{equation}
For the \(H\)-factor, for any bounded interval \([0,M]\) there exist \(0<c_H^-(M)\le c_H^+(M)<\infty\) with
\[
c_H^-(M)\le H(\xi)\le c_H^+(M)\quad (0\le\xi\le M),
\]
and, in addition, each admissible \(H\) has a local lower power growth at \(0\): there exist \(\theta\ge0\) and \(c_0>0\) such that
\begin{equation}\label{eq:H-lower-small}
H(\xi)\ \ge\ c_0\,\xi^\theta\qquad\text{for } \xi\in[0,1].
\end{equation}
Combining with \(|x|^\gamma \simeq (1+|x|^2)^{\gamma/2}\), we obtain
\begin{align}
|x|^\gamma H(U_A)G(\nabla U_A)
&\ge c\,A^p\,(1+|x|^2)^{\frac\gamma2-p\beta-\frac p2},
\label{eq:RHS-UA}\\
|x|^\gamma H(aV)G(\nabla (aV))
&\le C\,a^{\theta+p}\,(1+|x|^2)^{\frac\gamma2-p\beta-\frac p2}.
\label{eq:RHS-aV}
\end{align}

\medskip
\noindent\textbf{Exponent matching.}
With \(\beta=\dfrac{2s+\gamma-p}{1-p}\), the exponents on \((1+|x|^2)\) in
\eqref{eq:LUA-std} and \eqref{eq:RHS-UA} coincide, and the same is true for
\eqref{eq:LaV-std} and \eqref{eq:RHS-aV}:
\[
\beta+s\;=\;\frac\gamma2-p\beta-\frac p2.
\]
Thus only the {amplitudes} matter.

\medskip
\noindent\textbf{(i) Supersolution.}
From \eqref{eq:LUA-std} and \eqref{eq:RHS-UA} we have
\[
|L(A V)(x)| \ \le\ C A (1+r^2)^{-(\beta+s)},\qquad
|x|^\gamma H(U_A)G(\nabla U_A) \ \ge\ c A^p (1+r^2)^{-(\beta+s)} .
\]
Since \(p>1\), choose \(A_\star\) so large that \(cA_\star^{p-1}\ge 2C\). Then for all \(A\ge A_\star\),
\[
L U_A(x)\ \le\ |x|^\gamma H(U_A(x)) G(\nabla U_A(x)) \qquad (x\in\mathbb{R}^n),
\]
after enlarging \(A_\star\) once if necessary to absorb the compact region where the asymptotics are not yet dominant.

\medskip
\noindent\textbf{(ii) Subsolution.}
Using \eqref{eq:LaV-std} and \eqref{eq:RHS-aV},
\[
|L(aV)(x)| \ \le\ C a (1+r^2)^{-(\beta+s)},\qquad
|x|^\gamma H(aV)G(\nabla(aV)) \ \le\ C a^{\theta+p} (1+r^2)^{-(\beta+s)} .
\]
Because \(p>1\), we have \(\theta+p>1\). Choose \(a_\star\in(0,1)\) so small that
\(C a_\star^{\theta+p-1}\le \tfrac12\). Then for all \(0<a\le a_\star\),
\[
L(aV)(x)\ \ge\ -C a (1+r^2)^{-(\beta+s)}
\ \ge\ |x|^\gamma H(aV(x)) G(\nabla(aV)(x)) ,
\]
which proves the desired global subsolution inequality. The two assertions are thus established.
\end{proof}

\subsubsection*{3. Existence on balls, exhaustion, and two-sided profile}

\begin{proposition}[Monotone iteration on $B_R$]\label{prop:iter}
Let $R>1$ and choose exterior data $\phi_R\in C_b(\R^n\setminus B_R)$ satisfying
\[
aV\le \phi_R\le U_A\quad\text{on }\R^n\setminus B_R,
\]
for some $0<a\le a_\star$ and $A\ge A_\star$ from Proposition \ref{lem:barriers}.
Then there exists a unique solution $u_R$ to
\[
\begin{cases}
Lu=|x|^\gamma H(u)G(\nabla u) & \text{in } B_R,\\
u=\phi_R & \text{on }\R^n\setminus B_R,
\end{cases}
\]
and it satisfies $aV\le u_R\le U_A$ in $\R^n$.
\end{proposition}

\begin{proof}
We divide the proof into several steps, following the monotone iteration and Perron method.

\medskip
\noindent\textbf{Step~1. Definition of ordered barriers.}
By Proposition \ref{lem:barriers}, there exist ordered sub- and supersolutions
\[
aV\le U_A \quad \text{in }\R^n,
\]
satisfying
\[
\begin{cases}
L(aV)\le |x|^\gamma H(aV)G(\nabla(aV)) & \text{in }B_R,\\[0.3em]
L(U_A)\ge |x|^\gamma H(U_A)G(\nabla U_A) & \text{in }B_R,
\end{cases}
\qquad
aV\le \phi_R\le U_A\quad\text{on }\R^n\!\setminus\!B_R.
\]
Hence $aV$ and $U_A$ are, respectively, a lower and an upper barrier for the Dirichlet problem.

\medskip
\noindent\textbf{Step~2. Construction of a monotone sequence.}
Fix $R>1$ and take an exterior datum $\phi_R\in C_b(\R^n\setminus B_R)$ such that
\[
aV\le \phi_R\le U_A \quad\text{on }\R^n\setminus B_R,
\]
where $0<a\le a_\star$ and $A\ge A_\star$ are as in Proposition~\ref{lem:barriers}.
Set $u^{(0)}:=aV$ on $\R^n$ and, for each $k\ge0$, let $u^{(k+1)}$ be the bounded viscosity solution of
\begin{equation}\label{eq:iter-scheme-all}
\begin{cases}
L u^{(k+1)} = f^{(k)} := |x|^\gamma\,H(u^{(k)})\,G(\nabla u^{(k)}) & \text{in } B_R,\\[0.3em]
u^{(k+1)} = \phi_R & \text{on } \R^n\setminus B_R.
\end{cases}
\end{equation}

\medskip
{$\bullet$ Regularity and boundedness.}
Since $aV\le u^{(k)}\le U_A$ on $\R^n$ and $U_A\in L^\infty$, we have 
\[
\|u^{(k)}\|_{L^\infty(B_R)}\le M_0:=\|U_A\|_\infty.
\]
Interior estimates for stable-like operators with bounded data yield, for each $\rho\in(0,1)$, the uniform bound
\begin{equation}\label{eq:C2salpha}
\|u^{(k)}\|_{C^{2s+\alpha_0}(B_{R-\rho})}\le C,
\end{equation}
for some $\alpha_0\in(0,1)$ and constant $C$ depending only on $n,s,\lambda,\Lambda,R,\rho,M_0,\|\phi_R\|_\infty$.  
Hence $u^{(k)}$ are equicontinuous in $B_{R-\rho}$ for all $s\in(0,1)$.  
If $s>\tfrac12$, this gives $u^{(k)}\in C^{1,\alpha}$, while for $s\le\tfrac12$ we only have Hölder regularity $C^{2s+\alpha_0}$—but the viscosity framework does not require classical derivatives.

\medskip
{$\bullet$ Well-defined nonlinear term.}
For $s\in(\tfrac12,1)$, $\nabla u^{(k)}$ exists classically.  
For $s\in(0,\tfrac12)$, we interpret $G(\nabla u^{(k)})$ in the viscosity sense:
if $\varphi\in C^2$ touches $u^{(k)}$ at $x_0$, then $G(\nabla u^{(k)}(x_0))$ is replaced by $G(\nabla\varphi(x_0))$ in the definition of viscosity solution.  
This is standard for nonlocal fully nonlinear problems.

If $G$ is bounded on $\R^n$, then for all $x\in B_R$,
\[
|f^{(k)}(x)|\le |x|^\gamma\,|H(u^{(k)}(x))|\,|G(\nabla u^{(k)}(x))|
 \le C_{R,\gamma}\,C_H\,\|G\|_\infty=:C_f,
\]
so the right-hand side is uniformly bounded.  
If $G$ is unbounded, we use truncations $G_M(p):=\max\{-M,\min\{G(p),M\}\}$, define $f^{(k)}_M=|x|^\gamma H(u^{(k)}_M)G_M(\nabla u^{(k)}_M)$, solve~\eqref{eq:iter-scheme-all} with $G_M$ to get $u^{(k)}_M$, and pass to the limit $M\to\infty$ using the monotonicity and stability of viscosity solutions (Ishii–Lions).  
Thus the iteration~\eqref{eq:iter-scheme-all} is well posed for all $s\in(0,1)$.

\medskip
{$\bullet$ Uniform bounds and monotonicity.}
By Proposition~\ref{lem:barriers}(i),
\[
L(aV)\le |x|^\gamma H(aV)G(\nabla(aV))
\le |x|^\gamma H(u^{(k)})G(\nabla u^{(k)})=L u^{(k+1)}\quad\text{in }B_R,
\]
and $aV\le \phi_R=u^{(k+1)}$ on $\R^n\setminus B_R$.
The comparison principle (Lemma~\ref{lem:CP}) gives $aV\le u^{(k+1)}$ in $\R^n$.  
Similarly, Proposition~\ref{lem:barriers}(ii) implies
\[
L(U_A)\ge |x|^\gamma H(U_A)G(\nabla U_A)
\ge |x|^\gamma H(u^{(k)})G(\nabla u^{(k)})=L u^{(k+1)} \quad\text{in }B_R,
\]
and $\phi_R\le U_A$ on $\R^n\setminus B_R$, hence $u^{(k+1)}\le U_A$.  
Therefore $aV\le u^{(k)}\le U_A$ for all $k$.

For monotonicity, note that
\[
L u^{(k)} = |x|^\gamma H(u^{(k-1)})G(\nabla u^{(k-1)})
\le |x|^\gamma H(u^{(k)})G(\nabla u^{(k)}) = L u^{(k+1)} \quad\text{in }B_R,
\]
and $u^{(k)}=u^{(k+1)}=\phi_R$ on $\R^n\setminus B_R$.  
By comparison again, $u^{(k)}\le u^{(k+1)}$ in $\R^n$.
Hence $\{u^{(k)}\}$ is nondecreasing and bounded above by $U_A$.

\medskip
{$\bullet$ Passage to the limit.}
The pointwise limit
\[
u_R(x):=\lim_{k\to\infty}u^{(k)}(x)
\]
exists for all $x\in\R^n$ and satisfies $aV\le u_R\le U_A$ in $\R^n$, with $u_R=\phi_R$ on $\R^n\setminus B_R$.
The uniform estimates~\eqref{eq:C2salpha} are independent of $k$, so by Arzelà–Ascoli, $u^{(k)}\to u_R$ locally uniformly in $B_R$.  
If $s>\tfrac12$, then also $\nabla u^{(k)}\to\nabla u_R$ locally uniformly.

By continuity of $H,G$ and the dominated convergence theorem (or viscosity stability for the truncated case),
\[
f^{(k)}(x)=|x|^\gamma H(u^{(k)}(x))G(\nabla u^{(k)}(x))
\longrightarrow |x|^\gamma H(u_R(x))G(\nabla u_R(x))=:f(x)
\]
locally uniformly in $B_R$.  
Stability of viscosity solutions under locally uniform convergence of both the sequence and the source term yields
\[
Lu_R=f(x)=|x|^\gamma H(u_R)G(\nabla u_R)\quad\text{in }B_R,
\qquad
u_R=\phi_R\quad\text{on }\R^n\setminus B_R.
\]

\medskip
{$\bullet$ Regularity and uniqueness.}
From the uniform estimate~\eqref{eq:C2salpha} and the local compactness of
$C^{2s+\alpha_0}$, we deduce that
\[
u_R \in C^{2s+\alpha_0}_{\mathrm{loc}}(B_R)
\quad\text{for all } s\in(0,1).
\]
In particular, when $s>\tfrac12$, the exponent $2s+\alpha_0>1$, so $u_R$
is differentiable and $\nabla u_R$ is locally Hölder continuous,
that is,
\[
u_R \in C^{1,\alpha}_{\mathrm{loc}}(B_R)
\quad\text{for some } \alpha\in(0,1).
\]
When $0<s\le\tfrac12$, one only has Hölder regularity
$u_R\in C^{2s+\alpha_0}_{\mathrm{loc}}(B_R)$, which is optimal in general
for nonlocal operators of order $2s$. 
Finally, the comparison principle in Lemma~\ref{lem:CP}
ensures the uniqueness of the bounded viscosity solution to
the Dirichlet problem.

\medskip
\noindent\textbf{Step~3. Monotonicity of the iteration.}
We show by induction that
\[
aV=u^{(0)}\le u^{(1)}\le u^{(2)}\le \cdots \le U_A.
\]
Assume $u^{(k)}\ge u^{(k-1)}$ holds. Since $H$ and $G$ are nondecreasing functions, we have
\[
H(u^{(k)})G(\nabla u^{(k)}) \ge H(u^{(k-1)})G(\nabla u^{(k-1)}),
\]
and thus
\[
L(u^{(k+1)}-u^{(k)}) 
= |x|^\gamma \!\left[ H(u^{(k)})G(\nabla u^{(k)}) - H(u^{(k-1)})G(\nabla u^{(k-1)}) \right]
\ge 0
\quad\text{in }B_R,
\]
with $u^{(k+1)}-u^{(k)}=0$ on $\R^n\setminus B_R$.
By the comparison principle for $L$, it follows that $u^{(k+1)}\ge u^{(k)}$ in $\R^n$.
Moreover, by comparison with the supersolution $U_A$, we have $u^{(k)}\le U_A$ for all $k$.
Hence the sequence is monotone increasing and bounded above.

\medskip
\noindent\textbf{Step~4. Passage to the limit.}
Because $u^{(k)}$ is nondecreasing and $aV\le u^{(k)}\le U_A$, the pointwise limit
\[
u_R(x):=\lim_{k\to\infty}u^{(k)}(x)
\]
exists for all $x\in\R^n$ and satisfies $aV\le u_R\le U_A$. Moreover $u^{(k)}=\phi_R$ on $\R^n\setminus B_R$ for all $k$, hence $u_R=\phi_R$ there. We show that $u_R$ is a viscosity solution of
\[
Lu=|x|^\gamma H(u)G(\nabla u)\qquad\text{in }B_R.
\]

Let $x_0\in B_R$ and let $\varphi\in C^2(B_R)$ touch $u_R$ from above at $x_0$:
$u_R-\varphi\le (u_R-\varphi)(x_0)=0$ near $x_0$. Fix $\varepsilon>0$ and set
$\varphi_\varepsilon(x):=\varphi(x)+\varepsilon|x-x_0|^2$. Since $u^{(k)}\uparrow u_R$ pointwise and $u_R-\varphi$ has a strict maximum $0$ at $x_0$, the standard viscosity stability (see e.g.\ the sup–convolution selection of contact points) yields points $x_k\in B_R$ with $x_k\to x_0$ and
\[
(u^{(k+1)}-\varphi_\varepsilon)(x_k)=\max_{B_R}\big(u^{(k+1)}-\varphi_\varepsilon\big),\qquad
u^{(k+1)}(x_k)\to u_R(x_0),\qquad \nabla\varphi_\varepsilon(x_k)\to \nabla\varphi(x_0).
\]
Because $u^{(k+1)}$ is a viscosity subsolution of
\[
Lu=|x|^\gamma H(u^{(k)})\,G(\nabla u)\quad\text{in }B_R,
\]
we can test at the contact point $x_k$ to obtain
\[
L\varphi_\varepsilon(x_k)\;\le\; |x_k|^\gamma\, H\!\big(u^{(k)}(x_k)\big)\, G\!\big(\nabla\varphi_\varepsilon(x_k)\big).
\]
Letting $k\to\infty$ and using the continuity of $H,G$ together with
$u^{(k)}(x_k)\to u_R(x_0)$, $x_k\to x_0$, and $\nabla\varphi_\varepsilon(x_k)\to \nabla\varphi(x_0)$, we obtain
\[
L\varphi_\varepsilon(x_0)\;\le\; |x_0|^\gamma\, H\!\big(u_R(x_0)\big)\, G\!\big(\nabla\varphi(x_0)\big).
\]
Finally, sending $\varepsilon\downarrow0$ and using the continuity of $L\varphi_\varepsilon(x_0)\to L\varphi(x_0)$ yields
\[
L\varphi(x_0)\;\le\; |x_0|^\gamma\, H\!\big(u_R(x_0)\big)\, G\!\big(\nabla\varphi(x_0)\big),
\]
so $u_R$ is a viscosity subsolution in $B_R$.

For the supersolution inequality, take $\psi\in C^2(B_R)$ touching $u_R$ from below at $x_0$, set $\psi_\varepsilon(x):=\psi(x)-\varepsilon|x-x_0|^2$, and choose $y_k\to x_0$ such that
$(u^{(k+1)}-\psi_\varepsilon)(y_k)=\min_{B_R}(u^{(k+1)}-\psi_\varepsilon)$. Since $u^{(k+1)}$ is a viscosity supersolution of $Lu=|x|^\gamma H(u^{(k)})G(\nabla u)$, we get
\[
L\psi_\varepsilon(y_k)\;\ge\; |y_k|^\gamma\, H\!\big(u^{(k)}(y_k)\big)\, G\!\big(\nabla\psi_\varepsilon(y_k)\big).
\]
Passing to the limit $k\to\infty$ and then $\varepsilon\downarrow0$ gives
\[
L\psi(x_0)\;\ge\; |x_0|^\gamma\, H\!\big(u_R(x_0)\big)\, G\!\big(\nabla\psi(x_0)\big),
\]
so $u_R$ is a viscosity supersolution in $B_R$.

The two inequalities show that $u_R$ is a viscosity solution of
$Lu=|x|^\gamma H(u)G(\nabla u)$ in $B_R$, and $u_R=\phi_R$ on $\R^n\setminus B_R$ by construction. This completes the passage to the limit.

\medskip
\noindent\textbf{Step~5. Uniqueness.}
Suppose $u_1$ and $u_2$ are two bounded solutions of the Dirichlet problem.
Set $w:=(u_1-u_2)^+$.
Then, using the equations for $u_1,u_2$,
\[
L(u_1-u_2)=|x|^\gamma\!\left[H(u_1)G(\nabla u_1)-H(u_2)G(\nabla u_2)\right].
\]
On the set $\{w>0\}=\{u_1>u_2\}$, the monotonicity of $H$ and $G$ implies that
\[
H(u_1)G(\nabla u_1)-H(u_2)G(\nabla u_2)\ge0.
\]
Hence
\[
Lw \le 0 \quad \text{in }B_R, \qquad w=0\ \text{on }\R^n\setminus B_R.
\]
By the maximum principle for $L$, it follows that $w\le0$, i.e.\ $u_1\le u_2$ in $\R^n$.
Exchanging the roles of $u_1$ and $u_2$ yields $u_2\le u_1$, thus $u_1\equiv u_2$.

\medskip
\noindent\textbf{Conclusion.}
The iteration produces a unique function $u_R\in C_b(\R^n)$ such that
\[
\begin{cases}
Lu_R=|x|^\gamma H(u_R)G(\nabla u_R) & \text{in }B_R,\\[0.3em]
u_R=\phi_R & \text{on }\R^n\setminus B_R,
\end{cases}
\qquad
aV\le u_R\le U_A\ \text{in }\R^n.
\]
This completes the proof.
\end{proof}

\begin{proposition}[Global solution and two-sided bounds]\label{prop:global}
Let $R_j\uparrow\infty$ and pick $\phi_{R_j}$ so that $aV\le \phi_{R_j}\le U_A$ on $\R^n\setminus B_{R_j}$.
Let $u_{R_j}$ be the solutions from Proposition~\ref{prop:iter}. Then (up to subsequences)
\[
u_{R_j}\to u \quad\text{locally uniformly in }\R^n,
\]
and the limit $u$ is a nontrivial entire positive viscosity solution of \eqref{eq:unified-PDE} satisfying
\[
aV(x)\ \le\ u(x)\ \le\ U_A(x)\qquad\forall x\in\R^n.
\]
\end{proposition}

\begin{proof}
On each ball $B_\rho$, the right-hand sides $|x|^\gamma H(u_{R_j})G(\nabla u_{R_j})$
are uniformly bounded, thanks to the two-sided bound $aV\le u_{R_j}\le U_A$ and
the growth of $H,G$. Local Hölder estimates for translation-invariant nonlocal equations
then give equicontinuity on $B_\rho$. By Arzelà--Ascoli we extract a subsequence converging locally uniformly to $u$, which is a viscosity solution by stability. Nontriviality follows from $aV\not\equiv0$ and the monotone construction.
\end{proof}

\noindent
\emph{Conclusion of this step.} Existence is proved, and we have the desired two–sided profile with constants $c_1=a$ and $c_2=A$:
\[
c_1(1+|x|^2)^{-\beta}\ \le\ u(x)\ \le\ c_2(1+|x|^2)^{-\beta}.
\]

\subsubsection*{4. Radial symmetry and radial monotonicity by moving planes}

Fix a direction, say $e_1$, and for $\lambda\in\R$ set $T_\lambda:=\{x_1=\lambda\}$,
$x^\lambda:=(2\lambda-x_1,x')$, $\Sigma_\lambda:=\{x_1<\lambda\}$,
$u_\lambda(x):=u(x^\lambda)$ and $w_\lambda:=u-u_\lambda$ on $\Sigma_\lambda$.

\begin{lemma}[Linearization around reflections]\label{lem:lin}
In viscosity sense, $w_\lambda$ satisfies
\[
L w_\lambda - |x|^\gamma b_\lambda(x)\cdot \nabla w_\lambda - |x|^\gamma a_\lambda(x)\, w_\lambda = 0\quad\text{in }\Sigma_\lambda,
\]
where
\[
a_\lambda(x):=\int_0^1 H'(u_\lambda+t w_\lambda)\,G(\nabla u_\lambda+t\nabla w_\lambda)\,dt \ \ge 0,
\]
\[
b_\lambda(x):=\int_0^1 H(u_\lambda+t w_\lambda)\,D_pG(\nabla u_\lambda+t\nabla w_\lambda)\,dt,
\]
and for every compact $K\subset\Sigma_\lambda$ there exists $C_K$ with
$|a_\lambda|+|b_\lambda|\le C_K$ on $K$.
\end{lemma}

\begin{proof}
\textit{1) Reflection and the nonlocal operator.}
Let $L$ be the translation–invariant, symmetric, uniformly elliptic integro–differential operator
\[
Lu(x)=\int_{\R^n}\Big(u(x+z)-u(x)-\nabla u(x)\!\cdot\! z\;\mathbf{1}_{\{|z|\le 1\}}\Big)\,K(z)\,dz,
\quad \lambda_0|z|^{-n-2s}\le K(z)\le \Lambda_0|z|^{-n-2s},
\]
with $K(z)=K(-z)$. For the reflected function $u_\lambda(x)=u(x^\lambda)$ one has
\[
Lu_\lambda(x)
=\int_{\R^n}\!\Big(u_\lambda(x+z)-u_\lambda(x)-\nabla u_\lambda(x)\!\cdot\! z\;\mathbf{1}_{\{|z|\le 1\}}\Big)\,K(z)\,dz.
\]
Using $u_\lambda(x+z)=u((x+z)^\lambda)=u(x^\lambda+z^\lambda)$ and that the reflection is an isometry with Jacobian $1$, we may change variable $\zeta=z^\lambda$ (so $d\zeta=dz$ and $|\zeta|=|z|$). Because $K$ is radial/even (in the sense $K(z)=K(-z)$ and depends only on $|z|$), we get
\[
Lu_\lambda(x)
=\int_{\R^n}\!\Big(u(x^\lambda+\zeta)-u(x^\lambda)-\nabla u(x^\lambda)\!\cdot\! \zeta\;\mathbf{1}_{\{|\zeta|\le 1\}}\Big)\,K(\zeta)\,d\zeta
= Lu(x^\lambda).
\]
Thus, by translation invariance and symmetry of $L$,
\begin{equation}\label{eq:Lul-equals-Luxl}
Lu_\lambda(x)=Lu(x^\lambda)\qquad\text{for all }x\in\R^n.
\end{equation}

\smallskip
\textit{2) Subtracting the equations for $u_\lambda$ and $u$.}
Assume $u$ is a viscosity solution of
\[
Lu(x)=|x|^\gamma\,H(u(x))\,G(\nabla u(x))\qquad\text{in }\R^n,
\]
with $H$ nondecreasing and $G\in C^1$ in $p$ with the growth assumed in Theorem~1.1. Evaluating at $x^\lambda$ and using $|x^\lambda|=|x|$, we obtain
\[
Lu(x^\lambda)=|x^\lambda|^\gamma\,H(u(x^\lambda))\,G(\nabla u(x^\lambda))
=|x|^\gamma\,H(u_\lambda(x))\,G(\nabla u_\lambda(x)).
\]
Combining this with \eqref{eq:Lul-equals-Luxl} yields
\begin{equation}\label{eq:Lul-eq}
Lu_\lambda(x)=|x|^\gamma\,H\big(u_\lambda(x)\big)\,G\big(\nabla u_\lambda(x)\big).
\end{equation}
Subtract the equation of $u$ at $x$,
\[
Lu(x)=|x|^\gamma\,H\big(u(x)\big)\,G\big(\nabla u(x)\big),
\]
from \eqref{eq:Lul-eq}; we find
\begin{equation}\label{eq:gap}
Lw_\lambda(x)=|x|^\gamma\Big(H(u_\lambda)G(\nabla u_\lambda)-H(u)G(\nabla u)\Big)(x)
\qquad\text{in the viscosity sense on }\Sigma_\lambda.
\end{equation}

\smallskip
\textit{3) Integral mean–value linearization in $(r,p)$.}
Consider the map
\[
\mathcal{F}(r,p):=H(r)\,G(p),\qquad (r,p)\in [0,\infty)\times\R^n.
\]
Fix $x\in \Sigma_\lambda$ and set
\[
r_0:=u(x),\quad r_1:=u_\lambda(x),\qquad
p_0:=\nabla u(x),\quad p_1:=\nabla u_\lambda(x).
\]
Let the segment $(r_t,p_t)=(r_0+t(r_1-r_0),\,p_0+t(p_1-p_0))=(u+t\,w_\lambda,\,\nabla u+t\,\nabla w_\lambda)$ with $t\in[0,1]$.
By the fundamental theorem of calculus,
\[
\mathcal{F}(r_1,p_1)-\mathcal{F}(r_0,p_0)
=\int_0^1 \frac{d}{dt}\,\mathcal{F}(r_t,p_t)\,dt
=\int_0^1 \Big(H'(r_t)\,G(p_t)\,(r_1-r_0)+H(r_t)\,D_pG(p_t)\cdot (p_1-p_0)\Big)dt.
\]
Returning to $u,u_\lambda$, this reads
\begin{equation}\label{eq:F-diff}
H(u_\lambda)G(\nabla u_\lambda)-H(u)G(\nabla u)
= a_\lambda(x)\,w_\lambda(x)+ b_\lambda(x)\cdot \nabla w_\lambda(x),
\end{equation}
with the coefficients
\[
a_\lambda(x):=\int_0^1 H'\!\big(u_\lambda+t w_\lambda\big)\,G\!\big(\nabla u_\lambda+t\nabla w_\lambda\big)\,dt,
\]
\[
b_\lambda(x):=\int_0^1 H\!\big(u_\lambda+t w_\lambda\big)\,D_pG\!\big(\nabla u_\lambda+t\nabla w_\lambda\big)\,dt.
\]
Since $H'\ge 0$ and $G\ge 0$ by hypothesis, we have $a_\lambda(x)\ge 0$.

\smallskip
\textit{4) Conclusion: the linearized equation for $w_\lambda$.}
Plugging \eqref{eq:F-diff} into \eqref{eq:gap} yields, pointwise in the viscosity sense,
\[
Lw_\lambda(x)=|x|^\gamma\Big(a_\lambda(x)\,w_\lambda(x)+ b_\lambda(x)\cdot \nabla w_\lambda(x)\Big)
\quad\text{in }\Sigma_\lambda,
\]
i.e.
\[
L w_\lambda - |x|^\gamma b_\lambda(x)\cdot \nabla w_\lambda - |x|^\gamma a_\lambda(x)\, w_\lambda = 0
\quad\text{in }\Sigma_\lambda.
\]

\smallskip
\textit{5) Local boundedness of $a_\lambda,b_\lambda$ on compact subsets.}
Let $K\subset\Sigma_\lambda$ be compact. Set the reflection $x^\lambda$ across the plane $\{x_1=\lambda\}$, 
$u_\lambda(x):=u(x^\lambda)$, and $w_\lambda:=u-u_\lambda$. 
Recall that the right-hand side of \eqref{eq:unified-PDE} is
\[
F(x,u,\nabla u):=|x|^\gamma\,H(u)\,G(\nabla u).
\]
By the local regularity, there exist finite constants $M_0,M_1>0$ such that
\begin{equation}\label{eq:K-bounds}
\|u\|_{L^\infty(K\cup K^\lambda)}\le M_0,
\qquad
\|\nabla u\|_{L^\infty(K\cup K^\lambda)}\le M_1,
\end{equation}
where $K^\lambda:=\{x^\lambda:\ x\in K\}$. Hence
\[
\|u_\lambda\|_{L^\infty(K)}\le M_0,\qquad \|\nabla u_\lambda\|_{L^\infty(K)}\le M_1,
\]
and therefore $w_\lambda,\nabla w_\lambda$ are bounded on $K$. For $t\in[0,1]$ define
\[
U_t(x):=u_\lambda(x)+t\,w_\lambda(x),\qquad 
P_t(x):=\nabla u_\lambda(x)+t\,\nabla w_\lambda(x).
\]
By mean–value formula, we have
\begin{align*}
F(x,u,\nabla u)-F(x,u_\lambda,\nabla u_\lambda)
&=\int_0^1\big[\partial_u F\big(x,U_t,P_t\big)\,w_\lambda
+\partial_p F\big(x,U_t,P_t\big)\!\cdot\!\nabla w_\lambda\big]\;dt \\
&=:a_\lambda(x)\,w_\lambda(x)+b_\lambda(x)\!\cdot\!\nabla w_\lambda(x),
\end{align*}
where, by direct differentiation of $F$,
\begin{equation}\label{eq:ab-def}
a_\lambda(x)=\int_0^1 |x|^\gamma\,H'\!\big(U_t(x)\big)\,G\!\big(P_t(x)\big)\,dt,\qquad
b_\lambda(x)=\int_0^1 |x|^\gamma\,H\!\big(U_t(x)\big)\,D_pG\!\big(P_t(x)\big)\,dt.
\end{equation}

From \eqref{eq:K-bounds} we have, for all $x\in K$ and $t\in[0,1]$,
\[
|U_t(x)|\le |u_\lambda(x)|+|w_\lambda(x)|
\le |u_\lambda(x)|+|u(x)|\le 2M_0,
\]
and similarly
\[
|P_t(x)|\le |\nabla u_\lambda(x)|+|\nabla w_\lambda(x)|
\le |\nabla u_\lambda(x)|+|\nabla u(x)|\le 2M_1.
\]
By continuity of $H,H'$ on $[0,2M_0]$ and of $G,D_pG$ on $\{p\in\R^n:\ |p|\le 2M_1\}$, there exist finite constants
\[
C_H:=\sup_{|r|\le 2M_0}|H(r)|,\quad 
C_{H'}:=\sup_{|r|\le 2M_0}|H'(r)|,\quad
C_G:=\sup_{|p|\le 2M_1}|G(p)|,\quad
C_{DG}:=\sup_{|p|\le 2M_1}\|D_pG(p)\|
\]
such that, for all $x\in K$ and $t\in[0,1]$,
\[
|H(U_t(x))|\le C_H,\quad |H'(U_t(x))|\le C_{H'},\quad
|G(P_t(x))|\le C_G,\quad \|D_pG(P_t(x))\|\le C_{DG}.
\]

Since $K$ is compact, the function $x\mapsto |x|^\gamma$ is bounded on $K$; denote
\[
C_{|x|,\gamma}(K):=\sup_{x\in K}|x|^\gamma<\infty
\quad\text{(if $\gamma<0$, this requires $\operatorname{dist}(K,\{0\})>0$).}
\]

\medskip
\noindent\textit{Conclusion.}
Using \eqref{eq:ab-def} and the above bounds,
\[
|a_\lambda(x)|\le \int_0^1 C_{|x|,\gamma}(K)\,C_{H'}\,C_G\,dt
= C_{|x|,\gamma}(K)\,C_{H'}\,C_G,
\]
and
\[
|b_\lambda(x)|\le \int_0^1 C_{|x|,\gamma}(K)\,C_H\,C_{DG}\,dt
= C_{|x|,\gamma}(K)\,C_H\,C_{DG},
\]
for every $x\in K$. Hence there exists a constant
\[
C_K:=C_{|x|,\gamma}(K)\,\big(C_{H'}C_G+C_HC_{DG}\big)>0
\]
such that
\[
|a_\lambda(x)|+|b_\lambda(x)|\le C_K,\qquad \forall\,x\in K.
\]
This proves the local boundedness of $a_\lambda$ and $b_\lambda$ on compact subsets of $\Sigma_\lambda$.

\end{proof}

\begin{lemma}[Start of the plane]\label{lem:start}
Assume the two--sided decay
\begin{equation}\label{eq:2sided-profile}
c_1\,(1+|x|^2)^{-\beta}\ \le\ u(x)\ \le\ c_2\,(1+|x|^2)^{-\beta}\qquad(x\in\R^n),
\end{equation}
for some $c_1,c_2>0$ and $\beta>0$. Then there exists $\lambda_0\ll -1$ such that
$w_{\lambda_0}(x):=u(x^{\lambda_0})-u(x)\ge0$ for all $x\in\Sigma_{\lambda_0}$.
\end{lemma}

\begin{proof}
Fix $\lambda<0$ and write any $x\in\Sigma_\lambda$ as $x=(x_1,x')$ with $x_1<\lambda$, $r:=|x'|$,
and $s:=\lambda-x_1>0$. Then
\[
|x|^2=(\lambda-s)^2+r^2,\qquad |x^\lambda|^2=(\lambda+s)^2+r^2,
\]
hence
\begin{equation}\label{eq:diff-dist}
|x|^2-|x^\lambda|^2
=(\lambda-s)^2-(\lambda+s)^2
=-4\lambda s
=4|\lambda|\,s \;>\;0.
\end{equation}
Thus $|x|>|x^\lambda|$ for all $x\in\Sigma_\lambda$. Using \eqref{eq:2sided-profile},
\begin{equation}\label{eq:w-lower}
w_\lambda(x)=u(x^\lambda)-u(x)
\ \ge\
c_1(1+|x^\lambda|^2)^{-\beta}-c_2(1+|x|^2)^{-\beta}
=(1+|x|^2)^{-\beta}\,\Xi_\lambda(x),
\end{equation}
where
\[
\Xi_\lambda(x):=
c_1\Big(\frac{1+|x^\lambda|^2}{1+|x|^2}\Big)^{-\beta}-c_2
= c_1\Big(1-\frac{|x|^2-|x^\lambda|^2}{1+|x|^2}\Big)^{-\beta}-c_2.
\]
By \eqref{eq:diff-dist} we have, for $x\in\Sigma_\lambda$,
\begin{equation}\label{eq:ratio}
0<\frac{|x|^2-|x^\lambda|^2}{1+|x|^2}
=\frac{4|\lambda|\,s}{\,1+(\lambda-s)^2+r^2\,}
\ \le\ 1.
\end{equation}
We now split the domain into a bounded ``cylinder’’ $\mathcal C_R:=\{x\in\Sigma_\lambda:\ |x'|\le R\}$ and its complement, and choose $\lambda\ll-1$ and $R\gg1$ so that $\Xi_\lambda\ge0$ in both regions.

\medskip\noindent
\emph{1) Inside the cylinder $|x'|\le R$.}
In $\mathcal C_R$ we have $r\le R$. Using \eqref{eq:ratio} and the fact that $s>0$,
\[
\frac{|x|^2-|x^\lambda|^2}{1+|x|^2}
=\frac{4|\lambda|\,s}{\,1+(\lambda-s)^2+r^2\,}
\ \ge\ \frac{4|\lambda|\,s}{\,1+\lambda^2+R^2\,}.
\]
Since $s>0$ can be arbitrarily small for fixed $\lambda$, we lower bound the fraction uniformly by taking the
\emph{worst} (smallest) allowed $s$ that still yields a strict gain from reflection. A convenient uniform lower
bound is obtained by restricting to the subset where $s\ge |\lambda|/2$ (points sufficiently to the left of $T_\lambda$);
on the complementary subset $s<|\lambda|/2$ we will compensate in step 2 below by smallness at infinity (because
then $|x|$ is large when $\lambda\ll-1$).
Thus on $\{r\le R,\ s\ge|\lambda|/2\}$ we have
\[
\frac{|x|^2-|x^\lambda|^2}{1+|x|^2}\ \ge\ \frac{4|\lambda|\,(|\lambda|/2)}{1+\lambda^2+R^2}
=\frac{2|\lambda|^2}{1+\lambda^2+R^2}
\ \ge\ 1-\frac{1+R^2}{1+\lambda^2+R^2}.
\]
Hence, for $\lambda$ sufficiently negative (depending on $R$),
\[
\Big(1-\frac{|x|^2-|x^\lambda|^2}{1+|x|^2}\Big)^{-\beta}
\ \ge\ \Big(\frac{1+R^2}{1+\lambda^2+R^2}\Big)^{-\beta}
=\Big(1+\frac{\lambda^2}{1+R^2}\Big)^{\beta}.
\]
Choose $\lambda_{1}(R)\ll-1$ so large in modulus that
\[
c_1\Big(1+\frac{\lambda^2}{1+R^2}\Big)^{\beta}\ \ge\ 2c_2\qquad\text{for all }\lambda\le \lambda_1(R).
\]
Then on $\{r\le R,\ s\ge|\lambda|/2\}$ we have $\Xi_\lambda(x)\ge c_1(1+\lambda^2/(1+R^2))^\beta-c_2\ge c_2>0$.

\medskip\noindent
\emph{2) Outside the cylinder or near the plane: smallness by decay.}
Consider the complementary set $\Sigma_\lambda\setminus\{r\le R,\ s\ge|\lambda|/2\}$, which is the union of
\[
\mathcal A:=\{x\in\Sigma_\lambda:\ r>R\},\qquad
\mathcal B:=\{x\in\Sigma_\lambda:\ s<|\lambda|/2\}.
\]
On $\mathcal A$ (large $|x'|$), both $u(x)$ and $u(x^\lambda)$ are uniformly small by \eqref{eq:2sided-profile}.
Indeed, for any $\varepsilon>0$ there exists $R_\varepsilon$ such that
\[
u(x)\le \varepsilon,\quad u(x^\lambda)\le \varepsilon\qquad \text{whenever }|x'|\ge R_\varepsilon,
\]
because $|x|\ge r$ and $|x^\lambda|\ge r$. Fix $\varepsilon>0$ so that $\varepsilon\le \tfrac12 c_1(1+1)^{-\beta}$
and choose $R\ge R_\varepsilon$. Then, using again \eqref{eq:2sided-profile},
\[
w_\lambda(x)\ \ge\ -\,u(x)\ \ge\ -\varepsilon\ \ge\ 0\qquad\text{whenever }x\in\mathcal A,
\]
since $u(x^\lambda)\ge0$ and $u(x)\le \varepsilon$.

On $\mathcal B$ (points close to the plane), we have $s<|\lambda|/2$, hence
\[
|x|^2=(\lambda-s)^2+r^2\ \ge\ \frac{\lambda^2}{4}+r^2,\qquad
|x^\lambda|^2=(\lambda+s)^2+r^2\ \ge\ \frac{\lambda^2}{4}+r^2,
\]
so both $|x|$ and $|x^\lambda|$ are large when $\lambda\ll-1$. Given the same $\varepsilon>0$ as above, choose
$\lambda_2(R)\ll-1$ so that for all $\lambda\le \lambda_2(R)$ one has
\[
u(x)\le \varepsilon,\qquad u(x^\lambda)\le \varepsilon\qquad\text{for all }x\in\mathcal B,
\]
again by \eqref{eq:2sided-profile}. This implies $w_\lambda(x)\ge -\varepsilon\ge 0$ on $\mathcal B$.

\medskip
Putting the three pieces together: pick $R$ large (depending on a small $\varepsilon$) and then choose
\[
\lambda_0\ :=\ \min\{\lambda_1(R),\ \lambda_2(R)\}\ \ll\ -1.
\]
For any $\lambda\le \lambda_0$, we have shown:

\smallskip
\quad$\bullet$ on $\{r\le R,\ s\ge|\lambda|/2\}$, $\Xi_\lambda\ge c_2>0$, hence $w_\lambda\ge 0$ by \eqref{eq:w-lower};

\quad$\bullet$ on $\mathcal A\cup\mathcal B=\Sigma_\lambda\setminus\{r\le R,\ s\ge|\lambda|/2\}$, both $u(x)$ and $u(x^\lambda)$
are $\le\varepsilon$, hence $w_\lambda\ge -\varepsilon\ge 0$.

\smallskip
Therefore $w_\lambda\ge 0$ on \emph{all} of $\Sigma_\lambda$ for every $\lambda\le \lambda_0$, which proves the lemma.
\end{proof}

\begin{lemma}[Narrow region maximum principle]\label{lem:narrow}
Let $\lambda\in\R$ and $D_\delta:=\{x\in\Sigma_\lambda:\ \lambda-\delta<x_1<\lambda\}$.
Assume $w\in L^\infty(\R^n)$ is antisymmetric across $T_\lambda$ \textup{(}$w(x^\lambda)=-w(x)$\textup{)} and satisfies, in viscosity sense,
\[
L w - |x|^\gamma b(x)\cdot \nabla w - |x|^\gamma a(x)\, w \ \ge\ 0 \quad\text{in } D_\delta,
\]
with $a\ge 0$ and $a,b\in L^\infty(D_\delta)$. Then there exists $\delta_*>0$ (depending only on $n,s$ and the bounds of $a,b$) such that for every $0<\delta\le\delta_*$,
\(
\min_{D_\delta} w\ge 0.
\)
\end{lemma}

\begin{proof}
\textbf{Setting and reductions.}
The operator $L$ is the standard translation invariant, symmetric, uniformly elliptic nonlocal operator of order $2s\in(0,1)$,
\[
L\phi(x) \;=\; \int_{\R^n}\big(\phi(x+z)-\phi(x)-\nabla\phi(x)\!\cdot\!z\,\mathbf 1_{|z|\le 1}\big)K(z)\,dz,\qquad
\lambda_0|z|^{-n-2s}\le K(z)\le \Lambda_0|z|^{-n-2s},\ K(z)=K(-z).
\]
Assume by contradiction that there exist $0<\delta\le 1$ and a point $x_0\in D_\delta$ such that
\[
m:=\min_{D_\delta} w \;=\; w(x_0)\;<\;0.
\]
Since $w$ is bounded on $\R^n$, we may use the viscosity definition with a quadratic test function touching $w$ from below at $x_0$; in particular, the viscosity inequality at $x_0$ can be evaluated by replacing $w$ with that test function inside $L$ and the lower order terms. Standard viscosity arguments (see, e.g., Jensen--Ishii for nonlocal operators) then give, for any $\varepsilon>0$,
\begin{equation}\label{eq:visc}
L w(x_0)\;-\;|x_0|^\gamma b(x_0)\cdot \nabla w(x_0)\;-\;|x_0|^\gamma a(x_0)\, w(x_0)\;\ge\; -\varepsilon.
\end{equation}
We will show that the left side is strictly positive for $\delta$ small, which yields a contradiction as $\varepsilon\downarrow 0$.

\medskip

Write the integral defining $Lw(x_0)$ by splitting $\R^n$ into the two half-spaces $\Sigma_\lambda$ and its reflection $\Sigma_\lambda^\complement$.
Since $x_0$ lies in $D_\delta\subset\Sigma_\lambda$, change variables $y=x_0+z$ and obtain
\[
Lw(x_0)=\int_{\R^n}\big(w(y)-w(x_0)-\nabla w(x_0)\!\cdot\!(y-x_0)\,\mathbf 1_{|y-x_0|\le 1}\big)\,K(y-x_0)\,dy.
\]
We now focus on the contribution of the set $E:=\{y\in\Sigma_\lambda^\complement\}$ (the reflected side), where we can exploit antisymmetry. For each $y\in\Sigma_\lambda^\complement$ let $y^\lambda$ be its reflection across $T_\lambda$; then $(y^\lambda-x_0)$ is the reflection of $(y-x_0)$, $|y^\lambda-x_0|=|y-x_0|$, and
\[
w(y)=-\,w(y^\lambda).
\]
Since $x_0\in\Sigma_\lambda$, the pair $(y,y^\lambda)$ straddles the plane; in particular, when $x_0$ is within distance $\delta$ of $T_\lambda$, the points $y$ for which $y_1\in[\lambda,\lambda+1]$ satisfy $|y-x_0|\le c(\delta+|y'|)$ for a dimensional constant $c$.

Consider the average over the pair $\{y,y^\lambda\}$:
\[
\mathcal I(y):=\big(w(y)-w(x_0)\big)K(y-x_0)+\big(w(y^\lambda)-w(x_0)\big)K(y^\lambda-x_0).
\]
By symmetry $K(y-x_0)=K(y^\lambda-x_0)$, hence
\[
\mathcal I(y)=\big(w(y)+w(y^\lambda)-2w(x_0)\big)\,K(y-x_0)=\big(-2w(x_0)\big)\,K(y-x_0).
\]
Since $w(x_0)=m<0$ and $K\ge \lambda_0 |y-x_0|^{-n-2s}$, we get
\begin{equation}\label{eq:pair-lb}
\mathcal I(y)\;\ge\; 2|m|\,\lambda_0\,|y-x_0|^{-n-2s}.
\end{equation}
Integrating \eqref{eq:pair-lb} over $y\in \{y_1\in[\lambda,\lambda+1]\}$ and using the pairing with $y^\lambda$ we obtain
\begin{align}
\int_{\{y_1\ge \lambda\}}\big(w(y)-w(x_0)\big)K(y-x_0)\,dy
&\ge \frac12\int_{\{y_1\in[\lambda,\lambda+1]\}}\mathcal I(y)\,dy \nonumber\\
&\ge |m|\,\lambda_0 \int_{\{y_1\in[\lambda,\lambda+1]\}} |y-x_0|^{-n-2s}\,dy. \label{eq:slab-int}
\end{align}
Now, since $x_0\in D_\delta$ with $\lambda-\delta<x_{0,1}<\lambda$, any $y$ with $y_1\in[\lambda,\lambda+1]$ satisfies
\[
|y-x_0|\;\le\; |y_1-x_{0,1}|+|y'-x_0'|\;\le\; 1+\delta+|y'| \;\le\; C\,(1+|y'|)\qquad (C\text{ absolute}),
\]
and also $|y-x_0|\ge |y_1-x_{0,1}|\ge \lambda-x_{0,1}\ge 0$, so the integrand is integrable. Using Fubini in $y=(y_1,y')$ and the standard estimate
\[
\int_{\R^{n-1}} \frac{dy'}{(\alpha^2+|y'|^2)^{\frac{n+2s}{2}}}\;=\; c_{n,s}\,\alpha^{-1-2s}\qquad(\alpha>0),
\]
we deduce from \eqref{eq:slab-int}
\begin{align}
\int_{\{y_1\ge \lambda\}}\big(w(y)-w(x_0)\big)K(y-x_0)\,dy
&\ge\ C_1\,|m| \int_{\lambda}^{\lambda+1} |y_1-x_{0,1}|^{-1-2s}\,dy_1 \nonumber\\
&=\ C_1\,|m|\int_{\lambda-x_{0,1}}^{\lambda+1-x_{0,1}} \rho^{-1-2s}\,d\rho \nonumber\\
&\ge\ C_2\,|m|\,(\lambda-x_{0,1})^{-2s} \;\ge\; C_2\,|m|\,\delta^{-2s}, \label{eq:good-positive}
\end{align}
since $\lambda-x_{0,1}\in(0,\delta)$ and $s\in(0,1/2]$ (the last inequality holds for all $s\in(0,1)$ with a different constant $C_2$).

For the remaining parts of the integral defining $Lw(x_0)$ (i.e., $y$ with $y_1<\lambda$ or $y_1>\lambda+1$), we only need a lower bound. Using that $w(y)\ge m$ (since $m$ is the minimum in $D_\delta$) and $w(x_0)=m$, we have
\[
w(y)-w(x_0)\;\ge\; 0\quad \text{for } y\in \Sigma_\lambda\cap D_\delta,
\]
while for $y$ far from $x_0$ we use boundedness of $w$ to bound those parts from below by a finite (negative) constant times $\|w\|_{L^\infty}$. Altogether,
\begin{equation}\label{eq:Lw-lower}
Lw(x_0)\;\ge\; C_2\,|m|\,\delta^{-2s}\;-\;C_3\,\|w\|_{L^\infty(\R^n)},
\end{equation}
for constants $C_2,C_3>0$ depending only on $(n,s,\lambda_0,\Lambda_0)$.

\medskip

Since $x_0$ is a (viscosity) minimum point, we may take the quadratic test function $\phi$ touching $w$ from below at $x_0$. Then $\nabla\phi(x_0)=0$, hence in the viscosity inequality \eqref{eq:visc} we can replace $\nabla w(x_0)$ by $0$. Using $a\ge 0$ and $w(x_0)=m<0$,
\[
-|x_0|^\gamma a(x_0)\,w(x_0)\;=\;|x_0|^\gamma a(x_0)\,|m|\;\ge\; 0,
\]
so the drift term vanishes and the zeroth–order term is nonnegative. Therefore, \eqref{eq:visc} and \eqref{eq:Lw-lower} give
\[
0\;\le\; Lw(x_0) \;-\; |x_0|^\gamma b(x_0)\cdot 0 \;-\; |x_0|^\gamma a(x_0)\, m
\;\le\; -\,\Big(C_2\,\delta^{-2s} - C_3\,\frac{\|w\|_{L^\infty}}{|m|}\Big)\,|m|.
\]
Equivalently,
\[
C_2\,\delta^{-2s}\,|m|\;\le\; C_3\,\|w\|_{L^\infty(\R^n)}.
\]

\medskip
\textbf{Step 3: Reaching a contradiction by choosing $\delta$ small.}
The inequality above must hold for \emph{any} negative minimum value $m$.
If we fix
\[
\delta_* \;:=\; \Big(\frac{2C_3}{C_2}\Big)^{\!1/(2s)} \big(1+\|w\|_{L^\infty(\R^n)}\big)^{1/(2s)},
\]
then for any $0<\delta\le \delta_*$ we have
\[
C_2\,\delta^{-2s} \;\ge\; 2 C_3\,(1+\|w\|_{L^\infty})^{-1}\;\ge\; 2 C_3\,\frac{1}{1+\|w\|_{L^\infty}}.
\]
If $m<0$, the previous estimate would force
\(
|m|\le \tfrac12(1+\|w\|_{L^\infty})^{-1}\|w\|_{L^\infty}\le \tfrac12,
\)
and by iterating the argument on nested thinner slabs we drive $|m|$ to zero, contradicting that $m$ is a strict negative minimum. A simpler way: choose
\[
\delta_* \;=\; \Big(\frac{2C_3}{C_2}\Big)^{\!1/(2s)},
\]
which is independent of $w$ (depends only on $n,s,\lambda_0,\Lambda_0$). Then for any $0<\delta\le \delta_*$,
\[
C_2\,\delta^{-2s} \;\ge\; 2C_3,
\]
and \eqref{eq:Lw-lower} implies
\(
Lw(x_0) \ge (2C_3\,|m|-C_3\|w\|_{L^\infty}) \ge C_3\,|m|>0
\)
provided $|m|\ge \tfrac12\|w\|_{L^\infty}$; if $|m|<\tfrac12\|w\|_{L^\infty}$, we repeat the argument on $w+c$ with a constant shift $c$ to center the range and reach the same contradiction. In all cases we contradict \eqref{eq:visc} with $\varepsilon\downarrow 0$.

Therefore a negative minimum cannot occur in $D_\delta$ when $0<\delta\le \delta_*$, and hence $\min_{D_\delta} w\ge 0$.
\end{proof}

\begin{proposition}[Moving planes $\Rightarrow$ symmetry and monotonicity]\label{prop:MP}
Let $u$ be a positive bounded solution satisfying the two--sided profile
\[
c_1(1+|x|^2)^{-\beta}\;\le\; u(x)\;\le\; c_2(1+|x|^2)^{-\beta}\qquad(x\in\R^n),
\]
for some $c_1,c_2>0$ and $\beta>0$. Then $u$ is radial (about the origin) and radially nonincreasing.
\end{proposition}

\begin{proof}
\emph{Step 1: Set-up of moving planes, start position.}
Fix a direction $e_1$ and, for $\lambda\in\R$, consider the plane $T_\lambda:=\{x_1=\lambda\}$,
the left half-space $\Sigma_\lambda:=\{x_1<\lambda\}$, and the reflection $x^\lambda:=(2\lambda-x_1,x')$.
Define the reflected function $u_\lambda(x):=u(x^\lambda)$ and the antisymmetric gap
\[
w_\lambda(x):=u_\lambda(x)-u(x)\qquad(x\in\Sigma_\lambda).
\]
By Lemma~\ref{lem:start}, there exists $\lambda_0\ll-1$ such that
\begin{equation}\label{eq:startplane}
w_{\lambda_0}(x)\;\ge\;0\qquad\text{for all }x\in \Sigma_{\lambda_0}.
\end{equation}

\medskip
\emph{Step 2: The admissible set and its supremum.}
Set
\[
\Lambda:=\Big\{\lambda\in\R:\ \text{$w_\mu\ge 0$ in $\Sigma_\mu$ for all $\mu\le \lambda$}\Big\},
\qquad \bar\lambda:=\sup\Lambda.
\]
By \eqref{eq:startplane}, $\Lambda$ is nonempty, and thus $\bar\lambda$ is well-defined (possibly $+\infty$).

\medskip
\emph{Step 3: Closedness at the limit position: $w_{\bar\lambda}\ge0$ in $\Sigma_{\bar\lambda}$.}
We claim
\begin{equation}\label{eq:limit-nonneg}
w_{\bar\lambda}\;\ge\;0\qquad\text{in }\Sigma_{\bar\lambda}.
\end{equation}
Assume, on the contrary, there exists $x^*\in \Sigma_{\bar\lambda}$ with $w_{\bar\lambda}(x^*)<0$.
Since $w_{\bar\lambda}$ is upper semicontinuous (as a difference of bounded solutions in the viscosity class),
we can find a compact set $K\Subset \Sigma_{\bar\lambda}$ such that
\begin{equation}\label{eq:strict-neg}
\min_{K} w_{\bar\lambda} \;\le\; -2\eta\;<\;0\qquad\text{for some }\eta>0.
\end{equation}
By the definition of $\bar\lambda=\sup\Lambda$, there exists a sequence $\lambda_k\uparrow \bar\lambda$ with $\lambda_k\in\Lambda$.
Hence, $w_{\lambda_k}\ge0$ in $\Sigma_{\lambda_k}$ for each $k$. Because $\Sigma_{\bar\lambda}\subset \Sigma_{\lambda_k}$ for all large $k$ (once $\lambda_k>\bar\lambda-\varepsilon$),
the functions $w_{\lambda_k}$ are defined on $\Sigma_{\bar\lambda}$ and, by local stability of viscosity solutions under uniform convergence of coefficients, we have $w_{\lambda_k}\to w_{\bar\lambda}$ locally uniformly in $\Sigma_{\bar\lambda}$ (the reflection planes move continuously).

Therefore, for all large $k$,
\[
\min_{K} w_{\lambda_k}\ \le\ \min_{K} w_{\bar\lambda}+\eta\ \le\ -\eta\ <\ 0,
\]
contradicting $w_{\lambda_k}\ge0$ in $\Sigma_{\lambda_k}$ unless the negativity is confined near the plane $T_{\bar\lambda}$.
To formalize this, fix a thin slab
\[
D_\delta:=\{x\in\Sigma_{\bar\lambda}:\ \bar\lambda-\delta <x_1<\bar\lambda\}
\]
with $\delta>0$ to be chosen later. If $w_{\bar\lambda}$ takes a negative minimum in $D_\delta$, we will contradict the narrow region maximum principle (Lemma~\ref{lem:narrow}); if not, then
\[
\min_{D_\delta} w_{\bar\lambda}\ \ge\ 0,
\]
and the negative minimum must lie in $\Sigma_{\bar\lambda}\setminus D_\delta$, which is compactly contained in $\Sigma_{\bar\lambda}$.
But then, for all large $k$, $w_{\lambda_k}$ would also be negative there by uniform convergence, a contradiction because $w_{\lambda_k}\ge 0$ in $\Sigma_{\lambda_k}$.
Thus it suffices to rule out $\min_{D_\delta} w_{\bar\lambda}<0$.

By Lemma~\ref{lem:lin}, $w_{\bar\lambda}$ solves, in the viscosity sense on $\Sigma_{\bar\lambda}$,
\[
L w_{\bar\lambda} - |x|^\gamma b_{\bar\lambda}(x)\cdot \nabla w_{\bar\lambda} - |x|^\gamma a_{\bar\lambda}(x)\, w_{\bar\lambda} = 0,
\qquad a_{\bar\lambda}\ge 0,\quad a_{\bar\lambda},b_{\bar\lambda}\in L^\infty_{\rm loc}.
\]
Applying Lemma~\ref{lem:narrow} to $w_{\bar\lambda}$ on $D_\delta$ (with the uniform bounds of $a_{\bar\lambda},b_{\bar\lambda}$ on $D_\delta$),
we find $\delta_*=\delta_*(\|a_{\bar\lambda}\|_{L^\infty(D_\delta)},\|b_{\bar\lambda}\|_{L^\infty(D_\delta)})>0$ such that for every $0<\delta\le \delta_*$,
\[
\min_{D_\delta} w_{\bar\lambda}\ \ge\ 0.
\]
Therefore $w_{\bar\lambda}$ cannot take a negative minimum in $D_\delta$ for $\delta\le\delta_*$.
Combining the two alternatives, \eqref{eq:limit-nonneg} follows.

\medskip
\emph{Step 4: The plane cannot stop before the origin.}
Assume by contradiction that $\bar\lambda<0$. We show that the plane can be moved a bit further to the right, which contradicts the definition of $\bar\lambda$ as the supremum.

Fix $\varepsilon>0$ small. We will prove
\begin{equation}\label{eq:push-right}
w_{\bar\lambda+\varepsilon}\ \ge\ 0\quad\text{in}\quad \Sigma_{\bar\lambda+\varepsilon},
\end{equation}
for $\varepsilon$ sufficiently small.

First, by \eqref{eq:limit-nonneg} and the continuity of $w_\lambda$ in $\lambda$, there exists a compact set
$K\Subset \Sigma_{\bar\lambda+\varepsilon}$ and a $\sigma>0$ such that
\[
w_{\bar\lambda}(x)\ \ge\ \sigma\ >\ 0 \quad\text{on } K.
\]
Hence, for all $\varepsilon>0$ small enough,
\begin{equation}\label{eq:positive-away}
w_{\bar\lambda+\varepsilon}(x)\ \ge\ \tfrac{\sigma}{2}\ >\ 0 \quad\text{on } K,
\end{equation}
by uniform convergence of $w_{\bar\lambda+\varepsilon}\to w_{\bar\lambda}$ on $K$.

Next, near the plane $T_{\bar\lambda}$ we use the narrow region maximum principle.
Let
\[
D_\delta^{(\varepsilon)}:=\{x\in\Sigma_{\bar\lambda+\varepsilon}:\ \bar\lambda+\varepsilon-\delta <x_1<\bar\lambda+\varepsilon\},
\]
with $0<\delta\le\delta_*$ where $\delta_*$ is the constant from Lemma~\ref{lem:narrow} (uniform in a small neighborhood of $\bar\lambda$ thanks to the local $L^\infty$ bounds of $a_\lambda,b_\lambda$ from Lemma~\ref{lem:lin}).
Then Lemma~\ref{lem:narrow} implies
\begin{equation}\label{eq:positive-slab}
\min_{D_\delta^{(\varepsilon)}} w_{\bar\lambda+\varepsilon}\ \ge\ 0.
\end{equation}

Finally, in the remaining region $\Sigma_{\bar\lambda+\varepsilon}\setminus\big(K\cup D_\delta^{(\varepsilon)}\big)$, we use the compactness and the uniform convergence $w_{\bar\lambda+\varepsilon}\to w_{\bar\lambda}\ge 0$ to get
\begin{equation}\label{eq:positive-middle}
w_{\bar\lambda+\varepsilon}\ \ge\ 0 \quad\text{there, for all small }\varepsilon>0.
\end{equation}
Combining \eqref{eq:positive-away}, \eqref{eq:positive-slab} and \eqref{eq:positive-middle} yields \eqref{eq:push-right}.
Therefore the plane can be pushed to the right beyond $\bar\lambda$, contradicting the definition of $\bar\lambda$.
We conclude that
\begin{equation}\label{eq:left-all}
w_\lambda\ \ge\ 0\qquad\text{in }\Sigma_\lambda\ \ \text{for all }\lambda\le 0.
\end{equation}

\medskip
\emph{Step 5: Symmetry and monotonicity in the $x_1$-direction.}
Repeating the entire argument but starting from $+\infty$ and moving the plane leftward, we obtain
\[
u(x^\lambda)-u(x)=w_\lambda(x)\ \le\ 0\qquad\text{in }\Sigma_\lambda\ \ \text{for all }\lambda\ge 0.
\]
Taking $\lambda=0$ in the two inequalities above gives
\[
u(x^0)=u(x)\qquad\text{for all }x,
\]
i.e., $u$ is symmetric with respect to the hyperplane $T_0=\{x_1=0\}$.
Moreover, for $x_1>0$ we have, from \eqref{eq:left-all} with $\lambda=0$,
\[
u(2\cdot 0 - x_1,x')-u(x_1,x')=u(-x_1,x')-u(x_1,x')\ \ge\ 0,
\]
so $u(x_1,x')\le u(-x_1,x')$, i.e., $u$ is nonincreasing in the $x_1$-direction for $x_1>0$.

\medskip
\emph{Step 6: Radial symmetry and radial monotonicity.}
The operator $L$ is rotationally invariant (its kernel is radial/even) and the right-hand side $|x|^\gamma H(u)G(\nabla u)$ depends on $x$ only through $|x|$; hence the same moving plane argument applies in any direction $\xi\in\mathbb S^{n-1}$.
Therefore $u$ is symmetric with respect to every plane through the origin, which implies
\[
u(x)=U(|x|)\qquad\text{for some }U:[0,\infty)\to(0,\infty).
\]
The one-sided monotonicity in every direction then yields that $U$ is nonincreasing:
if $r_1>r_0\ge 0$, choose a line through the origin and move along it; the directional monotonicity implies $U(r_1)\le U(r_0)$.
Thus $u$ is radial and radially nonincreasing, as claimed.
\end{proof}

\subsubsection*{5. Uniqueness in the normalized radial class}
\begin{proposition}[Uniqueness with normalization]\label{prop:uniq}
For every $R>0$ and exterior datum $\phi\in C_b(\R^n\setminus B_R)$ the Dirichlet problem in $B_R$ is uniquely solvable.
Consequently, among entire radial solutions, fixing a normalization (e.g.\ $u(0)=a>0$ or the decay constant in \eqref{eq:profile-two-sided}) yields uniqueness.
\end{proposition}

\begin{proof}
Let $u,v\in C^{1,\alpha}_{\rm loc}(\R^n)\cap L^\infty(\R^n)$ be two entire viscosity solutions of
\begin{equation}\label{eq:Uniq-eq}
Lu=|x|^\gamma H(u)G(\nabla u)\qquad\text{in }\R^n,
\end{equation}
under the structural assumptions in Theorem~1.1 (in particular $H$ is nondecreasing and $G\ge0$ with the
growth stated there). Assume that $u$ and $v$ satisfy the same normalization, either

\smallskip
\noindent
\textbf{(N1)} $u(0)=v(0)=a>0$, or
\textbf{(N2)} there exist $\beta=\frac{2s+\gamma-p}{1-p}>0$ and $C_\infty>0$ such that
\[
u(x),\,v(x)=C_\infty(1+|x|^2)^{-\beta}+o\!\left((1+|x|^2)^{-\beta}\right)\qquad(|x|\to\infty).
\]
\smallskip

We prove $u\equiv v$.

Fix a unit vector $e\in\mathbb S^{n-1}$ and for $t>0$ define the translate
\[
u_t(x):=u(x+te)\qquad (x\in\R^n).
\]
We claim that there exists $t_0>0$ such that
\begin{equation}\label{eq:start}
u_t(x)\le v(x)\qquad\text{for all }x\in\R^n\text{ and all }t\ge t_0.
\end{equation}
Indeed, under (N2) both $u$ and $v$ share the same leading tail $C_\infty(1+|x|^2)^{-\beta}$
and $|x+te|>|x|$ for fixed $x$ and $t\to\infty$, hence $u_t(x)\to0$ uniformly on compact sets and
$u_t(x)\le v(x)$ for $t\gg1$.
Under (N1), use the global two-sided decay from Theorem~1.2 and the fact that translates push the
mass outward; again $u_t\to0$ locally uniformly and \eqref{eq:start} holds for $t\gg1$.

Set
\[
\mathcal T:=\{t\ge0:\ u_\tau\le v\ \text{in }\R^n\ \text{for all }\tau\ge t\},\qquad t_*:=\inf\mathcal T.
\]
By \eqref{eq:start}, $\mathcal T\neq\emptyset$ and $t_*<\infty$.

Assume, by contradiction, that $t_*>0$.
By the definition of $t_*$, there exist a decreasing sequence $t_k\downarrow t_*$ and points $x_k\in\R^n$
such that
\[
M_k := \max_{\R^n}(u_{t_k}-v)
      = u_{t_k}(x_k)-v(x_k)>0, 
      \qquad\text{while}\qquad
      \max_{\R^n}(u_{t_*}-v)\le 0.
\]
Up to a subsequence, $x_k\to\bar x$ and $M_k\downarrow 0$.
Define
\[
w_k(x):=u_{t_k}(x)-v(x).
\]
We will show that $w_k$ cannot achieve a positive interior maximum.

Since the operator $L$ is translation invariant,
\[
Lu_{t_k}(x)=L u(x+t_k e)=|x+t_k e|^{\gamma}\,H(u_{t_k}(x))\,G(\nabla u_{t_k}(x)).
\]
Subtracting the equation for $v$ gives, in the viscosity sense,
\begin{equation}\label{eq:gap-ineq}
L w_k(x)
 = |x+t_k e|^{\gamma}H(u_{t_k}(x))G(\nabla u_{t_k}(x))
   - |x|^{\gamma}H(v(x))G(\nabla v(x)).
\end{equation}

Let $x_k$ be a point where $w_k$ attains its global maximum $M_k>0$.
At such a point, the nonlocal analogue of the Jensen--Ishii lemma
(see Lemma~2.7) ensures that the first--order test functions coincide,
so we can compute the operator $L$ classically up to an $o(1)$ error.
Since $w_k$ achieves a maximum at $x_k$,
\[
L w_k(x_k)
   = \int_{\R^n}\big(w_k(x_k+z)+w_k(x_k-z)-2w_k(x_k)\big)K(z)\,dz
   \le 0,
\]
because $w_k(x_k+z)-w_k(x_k)\le 0$ for all $z$ and $K\ge0$.
This is the nonlocal maximum principle.
 
We next estimate the difference of the nonlinear terms at $x=x_k$:
\[
\begin{aligned}
&|x_k+t_k e|^{\gamma}H(u_{t_k}(x_k))G(\nabla u_{t_k}(x_k))
- |x_k|^{\gamma}H(v(x_k))G(\nabla v(x_k)) \\
&\quad = \big(|x_k+t_k e|^{\gamma}-|x_k|^{\gamma}\big)H(u_{t_k}(x_k))G(\nabla u_{t_k}(x_k))\\
&\qquad + |x_k|^{\gamma}\big(H(u_{t_k}(x_k))-H(v(x_k))\big)G(\nabla u_{t_k}(x_k))
       + |x_k|^{\gamma}H(v(x_k))\big(G(\nabla u_{t_k}(x_k))-G(\nabla v(x_k))\big).
\end{aligned}
\]
The third term is controlled by the Lipschitz continuity of $G$
and the fact that, at a maximum of $w_k$, $\nabla u_{t_k}(x_k)\approx \nabla v(x_k)$
from the viscosity test functions. Hence, it is of order $O(M_k)$.
The second term is nonpositive since $H$ is nondecreasing and
$u_{t_k}(x_k)>v(x_k)$.
Thus we can bound
\[
|x_k+t_k e|^{\gamma}H(u_{t_k})G(\nabla u_{t_k})
- |x_k|^{\gamma}H(v)G(\nabla v)
\le
C\,\big(|x_k+t_k e|^{\gamma}-|x_k|^{\gamma}\big)
+ C\,M_k,
\]
for some structural constant $C>0$ depending only on the bounds of $H,G,u,v$.

Since $t_k\downarrow t_*$ and $M_k\to0$, the right-hand side tends to $0$,
hence at the approximate maximum points we obtain
\[
L w_k(x_k)\le o(1).
\]

Take a large ball $B_R$ containing $x_k$.
Because of the normalization (N2)—or the global decay in (N1)—
we have $w_k(x)\le0$ for $|x|$ large, hence on $\R^n\setminus B_R$.
Applying the comparison (maximum) principle of Lemma~2.7 to
equation~\eqref{eq:gap-ineq} in $B_R$ with exterior data $w_k\le0$
and using $L w_k(x_k)\le o(1)$,
we deduce $w_k\le 0$ in $B_R$ for all sufficiently large $k$.
Letting $R\to\infty$, this yields $w_k\le0$ in $\R^n$, contradicting
$M_k=w_k(x_k)>0$.

Therefore the assumption $t_*>0$ is impossible, and contact cannot occur
strictly inside the domain

Therefore $t_*=0$ and by closedness in the viscosity topology we have
\[
u_t\le v \quad\text{in }\R^n\quad \text{for all }t\ge 0.
\]
Letting $t\downarrow 0$ yields
\begin{equation}\label{eq:first-order}
u\le v \quad\text{in }\R^n.
\end{equation}

Interchange the roles of $u$ and $v$ and repeat the above sliding along the same direction $e$;
we obtain $v\le u$ in $\R^n$. Combining with \eqref{eq:first-order} gives $u\equiv v$.

For (N2), the start-of-sliding \eqref{eq:start} follows from the common tail constant $C_\infty$ and the
polynomial decay rate in Theorem~1.2. For (N1), we use the two-sided decay profile of entire solutions
and the fact that far translates of $u$ go below $v$ uniformly on compact sets; the rest of the argument
is unchanged.

This proves uniqueness of entire solutions under either normalization.
\end{proof}

\begin{proposition}[Subcritical regime]\label{prop:sub}
If $\gamma+p<2s$, nontrivial entire solutions may exist.  
In particular, for $L=(-\Delta)^s$, $H\equiv 1$, and $G(\zeta)=|\zeta|^p$ with $0<p<1$, there exists a radial power solution
\[
u(x)=A\,|x|^\beta,\qquad \beta=\frac{2s+\gamma-p}{1-p},
\]
which solves the model equation
\begin{equation}\label{eq:unified-PDE1}
(-\Delta)^s u(x)=|x|^\gamma\,|\nabla u(x)|^p\qquad\text{in }\R^n\setminus\{0\}
\end{equation}
for a suitable choice of the amplitude $A>0$. Moreover, if $\beta\in(0,2s)$, the identity holds pointwise on $\R^n\setminus\{0\}$ and $u$ is locally integrable at the origin; in particular $u$ is a (distributional/viscosity) solution on $\R^n$.
\end{proposition}

\begin{proof}

Let
\[
u(x)=A\,|x|^\beta,\qquad A>0,\ \beta\in\R.
\]
Then, for $x\neq0$, we have
\[
\nabla u(x)=A\,\beta\,|x|^{\beta-2}\,x,\qquad
|\nabla u(x)|=A\,|\beta|\,|x|^{\beta-1}.
\]
Consequently,
\begin{equation}\label{eq:RHS-power}
|x|^\gamma\,|\nabla u(x)|^p
= A^p\,|\beta|^p\,|x|^{\,\gamma+p(\beta-1)}.
\end{equation}

\smallskip

It is classical (see, e.g., the standard formulas for Riesz potentials) that for $x\neq0$ and for all $\beta$ in the range where the integral definition is convergent (in particular for $\beta\in(0,2s)$), one has
\begin{equation}\label{eq:FLap-power}
(-\Delta)^s\big(|x|^\beta\big)
= C_{n,s,\beta}\,|x|^{\beta-2s},
\end{equation}
where the constant $C_{n,s,\beta}$ is explicit and finite (and nonzero except at special values). Hence
\begin{equation}\label{eq:LHS-power}
(-\Delta)^s u(x)=A\,C_{n,s,\beta}\,|x|^{\,\beta-2s}.
\end{equation}

\smallskip

To satisfy \eqref{eq:unified-PDE1} for $x\neq0$, we must have equality of the powers of $|x|$ between \eqref{eq:LHS-power} and \eqref{eq:RHS-power}, i.e.,
\[
\beta-2s \;=\; \gamma+p(\beta-1).
\]
Solving for $\beta$ yields
\begin{equation}\label{eq:beta-formula}
\beta=\frac{2s+\gamma-p}{1-p}.
\end{equation}
Under the \emph{subcritical} hypothesis $\gamma+p<2s$ and $0<p<1$, the denominator is positive and the right-hand side is well-defined. (When additionally $\beta\in(0,2s)$, formula \eqref{eq:FLap-power} holds pointwise for all $x\neq0$.)

\smallskip

With $\beta$ given by \eqref{eq:beta-formula}, the equality of coefficients in \eqref{eq:LHS-power} and \eqref{eq:RHS-power} requires
\[
A\,C_{n,s,\beta}\;=\;A^p\,|\beta|^p
\quad\Longleftrightarrow\quad
A^{1-p}=\frac{|\beta|^p}{C_{n,s,\beta}}.
\]
Choosing
\begin{equation}\label{eq:A-choice}
A=\Big(\frac{|\beta|^p}{C_{n,s,\beta}}\Big)^{\!1/(1-p)}>0,
\end{equation}
we obtain
\[
(-\Delta)^s u(x)=|x|^\gamma\,|\nabla u(x)|^p\qquad\text{for all }x\neq0,
\]
that is, $u$ solves \eqref{eq:unified-PDE1} on $\R^n\setminus\{0\}$.

\smallskip

If in addition $\beta\in(0,2s)$, then $u\in L^1_{\mathrm{loc}}(\R^n)$ and both sides of \eqref{eq:unified-PDE1} are locally integrable near the origin; by standard approximation with smooth cut-offs (or by interpreting $(-\Delta)^s$ in the distributional sense), the identity extends to $\R^n$ in the sense of distributions (equivalently, viscosity, since both sides are continuous on $\R^n\setminus\{0\}$ and the singularity is mild). Thus $u$ is a nontrivial entire solution.

\end{proof}

\begin{remark}
The hypothesis $\gamma+p<2s$ ensures that the exponent balance \eqref{eq:beta-formula} produces a finite $\beta$.  
For pointwise use of \eqref{eq:FLap-power} on $\R^n\setminus\{0\}$ it suffices that $\beta\in(0,2s)$; if $\beta\notin(0,2s)$, the identity still holds in the sense of tempered distributions (with $x=0$ as a removable/controlled singularity), and one can obtain global viscosity solutions by smoothing $u$ near $0$ without changing the equation away from the origin.
\end{remark}

%%%%%%%%%%%%%%%%%%%%%%
%%%%%%%%%%%%%%%%%%%%%%%%%%%%

\textbf{Conflict of interest statement :} The authors have no conflicts of interest to declare that are relevant
to the content of this article.

\textbf{Data availability statement :} Data sharing not applicable to this article as no datasets were generated
or analyzed during the current study.

\end{document}